\documentclass[12pt]{amsart}
\usepackage{amsmath,amsthm,amsfonts,amssymb,mathrsfs}
\date{\today}

\usepackage[cp1251]{inputenc}         

\usepackage[ukrainian]{babel} 
\usepackage{amsmath,amsthm,amsfonts,amssymb,mathrsfs}
\usepackage{eucal}

\usepackage{color}

\usepackage{amssymb,cite}
\usepackage[colorlinks,plainpages,citecolor=magenta, linkcolor=blue, backref]{hyperref}

\usepackage{hyperref}

\newtheorem{theorem}{Теорема}

\newtheorem{proposition}{Твердження}
\newtheorem{corollary}{Наслiдок}
\newtheorem{lemma}{Лема}
\theoremstyle{definition}
\newtheorem{remark}{Зауваження}

\begin{document}

\title[Полiциклiчнi розширення напiвгруп]{Полiциклiчнi розширення напiвгруп}

\author[Олег~Гутік, Павло Хилинський]{Олег~Гутік, Павло Хилинський}
\address{Механіко-математичний факультет, Львівський національний університет ім. Івана Франка, Університецька 1, Львів, 79000, Україна}
\email{oleg.gutik@lnu.edu.ua,
ovgutik@yahoo.com, pavlo.khylynskyi@lnu.edu.ua}


\begin{abstract}
Вводимо поняття $\lambda$-поліциклічного розширення Брука-Рейлі моноїда $S$ із визначеним гомоморфізмом $\theta$, яке є аналогом розширення Брука-Рейлі моноїда $S$. Описуємо ідемпотенти напівгрупи $\left(\mathscr{P}_{\lambda}(\theta,S),*\right)$ та відношення Ґріна на $\left(\mathscr{P}_{\lambda}(\theta,S),*\right)$. Доводимо, що $\left(\mathscr{P}_{\lambda}(\theta,S),*\right)$ --- $0$-проста напівгрупа для довільної напівгрупи $S$. Знайдено необхідні та достатні умови на моноїд $S$ i гомоморфізм $\theta$, за виконання яких напівгрупа $\left(\mathscr{P}_{\lambda}(\theta,S),*\right)$ є регулярною, інверсною, $0$-біпростою, комбінаторною, конгруенц-простою, чи інверсною 0-E-унітарною. Також вивчаємо топологізації напівгрупи $\left(\mathscr{P}_{\lambda}(\theta,S),*\right)$.

\emph{\textbf{Ключові слова:}} Напівгрупа, поліциклічний моноїд, розширення, напівтопологічна напівгрупа, топологічна напівгрупа.

\bigskip
\noindent
\emph{Oleg Gutik, Pavlo Khylynskyi, \textbf{Pjlycyclic extensions of semigroups},} Visnyk Lviv. Univ. Ser. Mech. Math. \textbf{90} (2020), 20--47.

\smallskip
\noindent
In the paper we introduce a notion of the Bruck-Reilly $\lambda$-polycyclic extension of a monoid $S$ with a homomorphism $\theta$ which is an analogue of the Bruck-Reilly extension of a monoid $S$. We describe idempotens of the semigroup $\left(\mathscr{P}_{\lambda}(\theta,S),*\right)$ and Green's relations on $\left(\mathscr{P}_{\lambda}(\theta,S),*\right)$. It is proved that $\left(\mathscr{P}_{\lambda}(\theta,S),*\right)$ is a $0$-simple semigroup for any semigroup $S$. We find necessary and sufficient conditions on a monoid $S$ and a homomorphism $\theta$ under which  the semigroup $\left(\mathscr{P}_{\lambda}(\theta,S),*\right)$ is regular, inverse, $0$-bisimple, combinatorial, congruence free, or inverse 0-E-unitary. Also we study topologizations of the semigroup $\left(\mathscr{P}_{\lambda}(\theta,S),*\right)$.
\end{abstract}

\maketitle


\section{\textbf{Вступ. Термінологія та позначення}}\label{section-1}

Ми користуватимемося термінологією з \cite{Carruth-Hildebrant-Koch-1983, Carruth-Hildebrant-Koch-1986, Clifford-Preston-1961, Clifford-Preston-1967, Engelking-1989, Lawson-1998, Petrich1984, Ruppert-1984}. Якщо $f\colon X\to Y$ -- відображення, то для довільної точки $y\in Y$ через $f^{-1}(y)$ будемо позначати повний прообраз точки $y$ стосовно відображення $f$.

Напівгрупа -- це непорожня множина з визначеною на ній бінарною асоціативною операцією.

Якщо $S$ --- напівгрупа, то через $S^1$ позначатимемо $S$ з приєднаною одиницею та  відношення Ґріна $\mathscr{R}$, $\mathscr{L}$, $\mathscr{J}$, $\mathscr{D}$ і $\mathscr{H}$ на $S$ визначаються так:
\begin{align*}
    &\qquad a\mathscr{R}b \mbox{~тоді і лише тоді, коли~} aS^1=bS^1;\\
    &\qquad a\mathscr{L}b \mbox{~тоді і лише тоді, коли~} S^1a=S^1b;\\
    &\qquad a\mathscr{J}b \mbox{~тоді і лише тоді, коли~} S^1aS^1=S^1bS^1;
\end{align*}
\begin{align*}
    &\qquad \mathscr{D}=\mathscr{L}\circ\mathscr{R}= \mathscr{R}\circ\mathscr{L};\\
    &\qquad \mathscr{H}=\mathscr{L}\cap\mathscr{R}
\end{align*}
(див. \cite[\S2.1]{Clifford-Preston-1961} або \cite{Green-1951}). Також, через $H_S(1_S)$ будемо позначати \emph{групу одиниць} моноїда $S$, і в цьому випадку, очевидно, що $H_S(1_S)$ є $\mathscr{H}$-класом одиниці $1_S$ моноїда $S$.

Нехай $S$ --- напівгрупа. Через $E(S)$ позначатимемо \emph{множину ідемпотентів} в $S$. Напівгрупова операція визначає частковий порядок $\leq$ на $E(S)$
\begin{equation*}
e\leq f \quad \Leftrightarrow \quad ef=fe=e.
\end{equation*}
Цей порядок називається \emph{природним частковим порядком} на $E(S)$. Напівґратка --- це комутативна напівгрупа ідемпотентів. Через $Z(S)$ будемо позначати центр напівгрупи $S$, тобто $Z(S)=\left\{s\in S\colon sx=xs \hbox{~для всіх~} x\in S\right\}$.

Якщо $\mathfrak{C}$~--- конгруенція на напівгрупі $S$, то через $[s]_{\mathfrak{C}}$ позначатимемо клас еквівалентності $\mathfrak{C}$, який містить елемент $s\in S$.

Напівгрупа $S$ називається:
\begin{itemize}
  \item \emph{простою}, якщо $S$ не має власних двобічних ідеалів;
  \item \emph{$0$-простою}, якщо $S$ містить нуль і $S$ не має власних двобічних ідеалів відмінних від $\{0\}$;
  \item \emph{біпростою}, якщо $S$ містить єдиний $\mathscr{D}$-клас;
  \item \emph{$0$-біпростою}, якщо $S$ має нуль і $S$ містить два $\mathscr{D}$-класи: $\{0\}$ і $S\setminus\{0\}$;
  \item \emph{комбінаторною}, якщо усі $\mathscr{H}$-класи в $S$ є одноелементними;
  \item \emph{конгруенц-простою}, якщо $S$ має лише одиничну й універсальну конгруенції.
\end{itemize}

Нагадаємо, що на інверсній напівгрупі $S$ напівгрупова операція визначає част\-ковий порядок $\leq$ на $S$
\begin{equation*}
x\leq y\qquad \Longleftrightarrow \qquad \hbox{існує} \quad e\in E(S) \quad \hbox{такий, що} \quad x=ey.
\end{equation*}
Цей порядок називається \emph{природним частковим порядком} на $S$.

Інверсна напівгрупа $S$ називається:
\begin{itemize}
  \item Е-\emph{унітарною}, якщо з  $es\in E(S)$ випливає, що $s\in E(S)$ для довільних $e\in E(S)$ і $s\in S$ \cite{Saito-1965};
  \item $0$-Е-\emph{унітарною}, якщо $S$ містить нуль $0_S$ i з  $es\in E(S)$ випливає, що $s\in E(S)$ для довільних $e\in E(S)\setminus\{0_S\}$ і $s\in S$ \cite{Lawson-1999, Szendrei-1987}.
\end{itemize}

\emph{Напівтопологічною} (\emph{топологічною}) напівгрупою називається топологічний простір з нарізно неперервною (неперервною) напівгруповою операцією. Ін\-верс\-на топологічна напівгрупа з неперервною інверсією називається \emph{топологічною ін\-верс\-ною напівгрупою}.

Топологія $\tau$ на напівгрупі $S$ називається:
\begin{itemize}
  \item \emph{трансляційно неперервною}, якщо $(S,\tau)$ --- напівтопологічна напівгрупа;
  \item \emph{напівгруповою}, якщо $(S,\tau)$ --- топологічна напівгрупа.
\end{itemize}

Біциклічний моноїд $\mathscr{C}(p,q)$  -- це напівгрупа з одиницею $1$ породжена двома елементами $p$ і $q$, що задовольняє умову $pq=1$. На $\mathscr{C}(p,q)$ напівгрупова операція визначається так:
\begin{equation*}
q^{k}p^{l}\cdot q^{m}p^{n}=q^{k+m-\min\{l,m\}}p^{l+n-\min\{l,m\}}.
\end{equation*}
Біциклічний моноїд $\mathscr{C}(p,q)$ є комбінаторною, біпростою, $F$-інверсною напівгрупою \cite{Lawson-1998} і відіграє важливу роль в алгебричній теорії напівгруп і в теорії топологічних напівгруп. Зокрема, добре відомий результат Андерсона \cite{Andersen-1952} стверджує, що (0-)проста напівгрупа є цілком (0-)простою тоді і тільки тоді, коли вона не містить ізоморфної копії біциклічного моноїда.

У 1970 році Ніва та Перро запропонували таке узагальнення біциклічного моноїда (\cite{Lawson-1998}, \cite{Nivat-Perrot-1970}). Для будь-якого ненульового кардинала $\lambda$  \emph{поліциклічний моноїд} $P_{\lambda}$ --- це напівгрупа з нулем така, що
\begin{equation*}
P_{\lambda}=\left\langle\{p_{i}\}_{i\in\lambda},\{p_{i}^{-1}\}_{i\in\lambda}\mid p_{i}p_{i}^{-1}=1 \hbox{ і }  p_{i}p_{j}^{-1}=0 \hbox{ для } i\neq j\right\rangle.
\end{equation*}
Очевидно, що у випадку $\lambda=1$  напівгрупа $P_{1}$ ізоморфна біциклічному моноїду з приєднаним нулем.

Рональд Брук у монографії \cite{Bruck-1958}, використовуючи біциклічний моноїд, побудував конструкцію алгебричного занурення довільної напівгрупи $S$ у простий моноїд $\boldsymbol{B}(S)$ (див. \cite[\S8.3]{Clifford-Preston-1967}). Рейлі \cite{Reilly=1966} та Уорн \cite{Warne=1966} узагальнили конструкцію Брука, побудувавши розширення Брука--Рейлі $\boldsymbol{BR}(S,\theta)$, для описання структури біпростих регулярних $\omega$-напівгруп (див. \cite[розділ II]{Petrich1984}). Топологізації напівгруп Бука та Брука--Рейлі вивчалися в працях  \cite{Gutik-1994, Gutik-1997, Gutik-2018, Gutik-Pavlyk=2009}. Структуру топологічних інверсних локально компактних біпростих $\omega$-напівгруп вивчали в працях \cite{Selden=1975, Selden=1976, Selden=1977}. Ми будуємо та досліджуємо аналог розширення Брука--Рейлі напівгруп для $\lambda$-поліциклічного моноїда та досліджуємо його властивості.

Нехай $\lambda$ -- довільний ненульовий кардинал. Надалі, через $\lambda^{*}$ будемо позначати вільний моноїд над алфавітом $\lambda$, а через $\varepsilon$ -- порожнє слово в $\lambda^{*}$. Для будь-якого слова $a\in \lambda^{*}$  позначимо:
\begin{itemize}
  \item[] $|a|$  -- довжину слова $a$;
  \item[] $\operatorname{suff}(a)=\{b\in\lambda^{*}\colon  \hbox{існує слово }c\in \lambda^{*} \hbox{ таке, що } cb=a\}$  -- множину всіх су\-фік\-сів слова $a$;
  \item[] $\operatorname{suff^{o}}(a)=\{b\in\lambda^{*}\colon  \hbox{існує слово } c\in \lambda^{*}\setminus\{\varepsilon\} \hbox{ таке, що } cb=a\}$  -- множину всіх власних суфіксів слова $a$.
\end{itemize}

Нехай $S$ -- моноїд і $\theta\colon S\rightarrow H_{S}(1)$ -- гомоморфізм з $S$ у його групу одиниць $H_{S}(1)$. Множина $\mathscr{P}_{\lambda}(\theta,S)=(S\times (P_{\lambda}\setminus\{0\}))\sqcup\{\boldsymbol{0}\}$ з бінарною операцією
\begin{equation}\label{eq-1}
(s,a_{1}^{-1}a_{2})*(t,b_{1}^{-1}b_{2})=
 \begin{cases}
   \left(\theta^{|u|}(s)t,(ua_{1})^{-1}b_{2}\right), &\text{якщо існує слово~} u\in\lambda^{*}\\
                                                     &\qquad\text{таке, що~} b_{1}=ua_{2};\\
   \left(s\theta^{|v|}(t),a_{1}^{-1}vb_{2}\right), &\text{якщо існує слово~}  v\in\lambda^{*}\\
                                                     &\qquad\text{таке, що~} a_{2}=vb_{1};\\
   \boldsymbol{0}, &\text{в іншому випадку,}
 \end{cases}
\end{equation}
i
\begin{equation*}
(s,a_{1}^{-1}a_{2})*\mathbf{0}=\boldsymbol{0}*(s,a_{1}^{-1}a_{2})=\boldsymbol{0}*\boldsymbol{0}=\boldsymbol{0},
\end{equation*}
де $\theta^{n}(s)=\underbrace{\theta\circ\cdots\circ\theta}_{n}(s)$ для будь-якого натурального числа $n$ і $\theta^{0}(s)=s$ називається \emph{$\lambda$-поліциклічним розширенням Брука-Рейлі} моноїда $S$ з визначеним гомоморфізмом $\theta$.

Ми доводимо, що так визначена бінарна операція $*$ на $\mathscr{P}_{\lambda}(\theta,S)$ є асоціативною, а також описуємо ідемпотенти напівгрупи $\left(\mathscr{P}_{\lambda}(\theta,S),*\right)$ та відношення Ґріна на $\left(\mathscr{P}_{\lambda}(\theta,S),*\right)$. Доведено, що $\left(\mathscr{P}_{\lambda}(\theta,S),*\right)$ --- $0$-проста напівгрупа для довільної напівгрупи $S$. Знайдено необхідні та достатні умови на моноїд $S$ гомоморфізм $\theta$, за виконання яких напівгрупа $\left(\mathscr{P}_{\lambda}(\theta,S),*\right)$ є регулярною, інверсною, $0$-біпростою, комбінаторною, конгруенц-простою, чи інверсною 0-E-унітарною. Також вивчається топологізація напівгрупи $\left(\mathscr{P}_{\lambda}(\theta,S),*\right)$. Отримані результати анонсовано в \cite{PKhylynskyi-Gutik=2019}.

\section{\textbf{Алгебричні властивості напівгрупи $\mathscr{P}_{\lambda}(\theta,S)$}}

\begin{proposition}\label{proposition-1}
$\left(\mathscr{P}_{\lambda}(\theta,S),*\right)$ є напівгрупою.
\end{proposition}

\begin{proof}
Нехай $(s,a_{1}^{-1}a_{2})$, $(t,b_{1}^{-1}b_{2})$ і $(r,c_{1}^{-1}c_{2})$ -- довільні ненульові елементи множини $\mathscr{P}_{\lambda}(\theta,S)$. Розглянемо можливі випадки.

1. Існують слова $u,v\in\lambda^{*}$ такі, що $b_{1}=ua_{2}$ і $c_{1}=vb_{2}$. Тоді:
\begin{align*}
((s,a_{1}^{-1}a_{2})*(t,b_{1}^{-1}b_{2}))*(r,c_{1}^{-1}c_{2})&=
(\theta^{|u|}(s)t,(ua_{1})^{-1}b_{2})*(r,c_{1}^{-1}c_{2})=\\
&=(\theta^{|v|}(\theta^{|u|}(s)t)r,(vua_{1})^{-1}c_{2})=\\
&=(\theta^{|v|+|u|}(s)\theta^{|v|}(t)r,(vua_{1})^{-1}c_{2})=\\
&=(\theta^{|vu|}(s)\theta^{|v|}(t)r,(vua_{1})^{-1}c_{2})=\\
&=(s,a_{1}^{-1}a_{2})*(\theta^{|v|}(t)r,(vb_{1})^{-1}c_{2})=\\
&=(s,a_{1}^{-1}a_{2})*((t,b_{1}^{-1}b_{2})*(r,c_{1}^{-1}c_{2})).
\end{align*}

2. Існують слова $u,v\in\lambda^{*}$ такі, що $a_{2}=ub_{1}$ і $c_{1}=vb_{2}$. Тоді розглянемо можливі підвипадки.

a) Існує слово $w\in\lambda^{*}$ таке, що $u=wv$. Тоді:
\begin{align*}
((s,a_{1}^{-1}a_{2})*(t,b_{1}^{-1}b_{2}))*(r,c_{1}^{-1}c_{2})&=
(s\theta^{|u|}(t),a_{1}^{-1}ub_{2})*(r,c_{1}^{-1}c_{2})=\\
&=(s\theta^{|wv|}(t),a_{1}^{-1}wvb_{2})*(r,c_{1}^{-1}c_{2})=\\
&=(s\theta^{|wv|}(t),a_{1}^{-1}wc_{1})*(r,c_{1}^{-1}c_{2})=\\
&=(s\theta^{|wv|}(t)\theta^{|w|}(r),a_{1}^{-1}wc_{2})=\\
&=(s\theta^{|w|+|v|}(t)\theta^{|w|}(r),a_{1}^{-1}wc_{2}))=\\
&=(s\theta^{|w|}(\theta^{|v|}(t)r),a_{1}^{-1}wc_{2}))=\\
&=(s,a_{1}^{-1}a_{2})*(\theta^{|v|}(t)r,(vb_{1})^{-1}c_{2}))=\\
&=(s,a_{1}^{-1}a_{2})*((t,b_{1}^{-1}b_{2})*(r,c_{1}^{-1}c_{2})).
\end{align*}

b) Існує слово $w\in\lambda^{*}$ таке, що $v=wu$. Тоді:
\begin{align*}
((s,a_{1}^{-1}a_{2})*(s,b_{1}^{-1}b_{2}))*(r,c_{1}^{-1}c_{2})&=
(s\theta^{|u|}(t),a_{1}^{-1}ub_{2})*(r,c_{1}^{-1}c_{2})=\\
&=(\theta^{|w|}(s\theta^{|u|}(t))r,(wa_{1})^{-1}c_{2})=\\
&=(\theta^{|w|}(s)\theta^{|w|+|u|}(t)r,(wa_{1})^{-1}c_{2})=\\
&=(\theta^{|w|}(s)\theta^{|wv|}(t)r,(wa_{1})^{-1}c_{2})=\\
&=(\theta^{|w|}(s)\theta^{|v|}(t)r,(wa_{1})^{-1}c_{2})=\\
&=(s,a_{1}^{-1}a_{2})*(\theta^{|v|}(t)r,(wa_{2})^{-1}c_{2})=\\
&=(s,a_{1}^{-1}a_{2})*(\theta^{|v|}(t)r,(wub_{1})^{-1}c_{2})=\\
&=(s,a_{1}^{-1}a_{2})*(\theta^{|v|}(t)r,(vb_{1})^{-1}c_{2})=\\
&=(s,a_{1}^{-1}a_{2})*((t,b_{1}^{-1}b_{2})*(r,c_{1}^{-1}c_{2})).
\end{align*}

c) У випадку $u\notin\operatorname{suff}(v)$ і $v\notin\operatorname{suff}(u)$ отримуємо, що
\begin{align*}
((s,a_{1}^{-1}a_{2})*(t,b_{1}^{-1}b_{2}))*(r,c_{1}^{-1}c_{2})&=
(s\theta^{|u|}(t),a_{1}^{-1}ub_{2})*(r,c_{1}^{-1}c_{2})=\\
&=\boldsymbol{0}=\\
&=(s,a_{1}^{-1}a_{2})*(\theta^{|v|}(t)r,(vb_{1})^{-1}c_{2})=\\
&=(s,a_{1}^{-1}a_{2})*((t,b_{1}^{-1}b_{2})*(r,c_{1}^{-1}c_{2})).
\end{align*}

3. Існують слова $u,v\in\lambda^{*}$ такі, що $b_{1}=ua_{2}$ і $b_{2}=vc_{1}$. Тоді:
\begin{align*}
((s,a_{1}^{-1}a_{2})*(t,b_{1}^{-1}b_{2}))*(r,c_{1}^{-1}c_{2})&=
(\theta^{|u|}(s)t,(ua_{1})^{-1}b_{2})*(r,c_{1}^{-1}c_{2})=\\
&=(\theta^{|u|}(s)t\theta^{|v|}(r),(ua_{1})^{-1}vc_{2})=\\
&=(s,a_{1}^{-1}a_{2})*(t\theta^{|v|}(r),b_{1}^{-1}vc_{2})=\\
&=(s,a_{1}^{-1}a_{2})*((t,b_{1}^{-1}b_{2})*(r,c_{1}^{-1}c_{2})).
\end{align*}

4. Існують слова $u,v\in\lambda^{*}$ такі, що $a_{2}=ub_{1}$ і $b_{2}=vc_{1}$. Тоді:
\begin{align*}
((s,a_{1}^{-1}a_{2})*(t,b_{1}^{-1}b_{2}))*(r,c_{1}^{-1}c_{2})&=
(s\theta^{|u|}(t),a_{1}^{-1}ub_{2})*(r,c_{1}^{-1}c_{2})=\\
&=(s\theta^{|u|}(t)\theta^{|uv|}(r),a_{1}^{-1}uvc_{2})=\\
&=(s\theta^{|u|}(t)\theta^{|u|+|v|}(r),a_{1}^{-1}uvc_{2})=\\
&=(s\theta^{|u|}(t\theta^{|v|}(r)),a_{1}^{-1}uvc_{2})=\\
&=(s,a_{1}^{-1}a_{2})*(t\theta^{|v|}(r),b_{1}^{-1}vc_{2})=\\
&=(s,a_{1}^{-1}a_{2})*((t,b_{1}^{-1}b_{2})*(r,c_{1}^{-1}c_{2})).
\end{align*}

5. Існує слово $u\in\lambda^{*}$ таке, що $b_{1}=ua_{2}$, $b_{2}\notin\operatorname{suff}(c_{1})$ і $c_{1}\notin\operatorname{suff}(b_{2})$. Тоді:
\begin{align*}
((s,a_{1}^{-1}a_{2})*(t,b_{1}^{-1}b_{2}))*(r,c_{1}^{-1}c_{2})&=
(\theta^{|u|}(s)t,(ua_{1})^{-1}b_{2})*(r,c_{1}^{-1}c_{2})=\\
&=\boldsymbol{0}=\\
&=(s,a_{1}^{-1}a_{2})*((t,b_{1}^{-1}b_{2})*(r,c_{1}^{-1}c_{2})).
\end{align*}

6. Існує слово $u\in\lambda^{*}$ таке, що $c_{1}=ub_{2}$, $a_{2}\notin\operatorname{suff}(b_{1})$ і $b_{1}\notin\operatorname{suff}(a_{2})$. Тоді:
\begin{align*}
((s,a_{1}^{-1}a_{2})*(t,b_{1}^{-1}b_{2}))*(r,c_{1}^{-1}c_{2})&=\boldsymbol{0}=\\
&=(s,a_{1}^{-1}a_{2})*(\theta^{|u|}(t)r,(ub_{1})^{-1}c_{2})=\\
&=(s,a_{1}^{-1}a_{2})*((t,b_{1}^{-1}b_{2})*(r,c_{1}^{-1}c_{2})).
\end{align*}

7. Існує слово $u\in\lambda^{*}$ таке, що $a_{2}=ub_{1}$, $b_{2}\notin\operatorname{suff}(c_{1})$ і $c_{1}\notin\operatorname{suff}(b_{2})$.  Тоді:
\begin{align*}
((s,a_{1}^{-1}a_{2})*(t,b_{1}^{-1}b_{2}))*(r,c_{1}^{-1}c_{2})&=
(s\theta^{|u|}(t),a_{1}^{-1}ub_{2})*(r,c_{1}^{-1}c_{2})=\\
&=\boldsymbol{0}=\\
&=(s,a_{1}^{-1}a_{2})*((t,b_{1}^{-1}b_{2})*(r,c_{1}^{-1}c_{2})).
\end{align*}

8. Існує слово $u\in\lambda^{*}$ таке, що $b_{2}=uc_{1}$ і $a_{2}\notin\operatorname{suff}(b_{1})$ і $b_{1}\notin\operatorname{suff}(a_{2})$. Тоді
\begin{align*}
((s,a_{1}^{-1}a_{2})*(t,b_{1}^{-1}b_{2}))*(r,c_{1}^{-1}c_{2})&=\boldsymbol{0}=\\
&=(s,a_{1}^{-1}a_{2})*(t\theta^{|u|}(r),b_{1}^{-1}uc_{2})=\\
&=(s,a_{1}^{-1}a_{2})*((t,b_{1}^{-1}b_{2})*(r,c_{1}^{-1}c_{2})).
\end{align*}

9. У випадку $a_{2}\notin\operatorname{suff}(b_{1})$, $b_{1}\notin\operatorname{suff}(a_{2})$, $b_{2}\notin\operatorname{suff}(c_{1})$ і $c_{1}\notin\operatorname{suff}(b_{2})$ маємо, що
\begin{equation*}
((s,a_{1}^{-1}a_{2})*(t,b_{1}^{-1}b_{2}))*(r,c_{1}^{-1}c_{2})=\boldsymbol{0}=
(s,a_{1}^{-1}a_{2})*((t,b_{1}^{-1}b_{2})*(r,c_{1}^{-1}c_{2})).
\end{equation*}

Отож, бінарна операція $*$ на $\mathscr{P}_{\lambda}(\theta,S)$  асоціативна.
\end{proof}

Надалі, якщо не зазначено інше, то будемо вважати, що $S$~--- моноїд з одиницею $1_S$ і групою одиниць $H_S(1_S)$. Також через $1_{P_{\lambda}}$ i $0_{P_{\lambda}}$  позначатимемо одиницю та нуль поліциклічного моноїда $P_{\lambda}$, а через $\boldsymbol{0}$~--- нуль напівгрупи $\mathscr{P}_{\lambda}(\theta,S)$.

\begin{proposition}\label{proposition-2}
Ненульовий елемент $(s,a_{1}^{-1}a_{2})$ напівгрупи $\mathscr{P}_{\lambda}(\theta,S)$ є ідемпотентом тоді і тільки тоді, коли  $s\in E(S)$ і $a_{1}=a_{2}$.
\end{proposition}

\begin{proof}
$(\Rightarrow)$
Нехай  $(e,a_{1}^{-1}a_{2})$ -- ненульовий ідемпотент напівгрупи $\mathscr{P}_{\lambda}(\theta,S)$. Розглянемо мож\-ли\-ві випадки.

1. Якщо існує слово $u\in\lambda^{*}$ таке, що $a_{1}=ua_{2}$, то
\begin{align*}
(e,a_{1}^{-1}a_{2})=(e,a_{1}^{-1}a_{2})*(e,a_{1}^{-1}a_{2})=(\theta^{|u|}(e)\cdot e,(ua_{1})^{-1}a_{2}).
\end{align*}
З цієї рівності випливає, що $a_{1}=ua_{1}$. Тому $u=\varepsilon$, а отже,
\begin{equation*}
e=\theta^{|u|}(e)\cdot e=\theta^{0}(e)\cdot e=e\cdot e
\end{equation*}
і $a_{1}=a_{2}$.

2. Якщо існує слово $u\in\lambda^{*}$ таке, що $a_{2}=ua_{1}$, то
\begin{align*}
(e,a_{1}^{-1}a_{2})=(e,a_{1}^{-1}a_{2})*(e,a_{1}^{-1}a_{2})=(e\cdot \theta^{|u|}(e) ,a_{1}^{-1}ua_{2}).
\end{align*}
З цієї рівності випливає, що $a_{2}=ua_{2}$. Тому $u=\varepsilon$, а отже,
\begin{equation*}
e=e\cdot \theta^{|u|}(e)=e\cdot\theta^{0}(e)=e\cdot e
\end{equation*}
і $a_{1}=a_{2}$.

$(\Leftarrow)$
Нехай $e$ -- довільний ідемпотент напівгрупи $S$ і $a$ -- довільне слово з $\lambda^{*}$. Тоді $(e,a^{-1}a)*(e,a^{-1}a)=(e^{2},a^{-1}a)=(e,a^{-1}a)$.
\end{proof}

\begin{proposition}\label{proposition-3}
Ідемпотенти комутують у $\mathscr{P}_{\lambda}(\theta,S)$ тоді і тільки тоді, коли ідемпотенти комутують у напівгрупі $S$.
\end{proposition}

\begin{proof}
$(\Rightarrow)$
Нехай $e,f$ -- довільні ідемпотенти напівгрупи $S$. Тоді елементи $(e,1_{P_{\lambda}})$ і  $(f,1_{P_{\lambda}})$ є ідемпотентами напівгрупи $\mathscr{P}_{\lambda}(\theta,S)$. Оскільки індемпотенти в $\mathscr{P}_{\lambda}(\theta,S)$ комутують, то
\begin{equation*}
(e\cdot f,1_{P_{\lambda}})=(e,1_{P_{\lambda}})*(f,1_{P_{\lambda}})=(f,1_{P_{\lambda}})*(e,1_{P_{\lambda}})=(f\cdot e,1_{P_{\lambda}}),
\end{equation*}
а отже, $e\cdot f=f\cdot e$.

$(\Leftarrow)$
Нехай $(e,a^{-1}a)$ і $(f,b^{-1}b)$ -- довільні ненульові ідемпотенти напівгрупи $\mathscr{P}_{\lambda}(\theta,S)$. Розглянемо можливі випадки.

1. Існує слово $u\in\lambda^{*}$ таке, що $a=ub$. Тоді
\begin{align*}
(e,a^{-1}a)*(f,b^{-1}b)&=(e\cdot \theta^{|u|}(f),a^{-1}ub)=\\
&=(e\cdot 1_{S},a^{-1}ub)=\\
&=(1_{S}\cdot e ,a^{-1}ub)=\\
&=(\theta^{|u|}(f)\cdot e,(ub)^{-1}a)=\\
&=(f,b^{-1}b)*(e,a^{-1}a).
\end{align*}

2. Існує слово $u\in\lambda^{*}$ таке, що $b=ua$. Тоді
\begin{align*}
(e,a^{-1}a)*(f,b^{-1}b)&=(\theta^{|u|}(e)\cdot f,(ua)^{-1}b)=\\
&=(1_{S}\cdot f ,a^{-1}ub)=\\
&=(f\cdot 1_{S},a^{-1}ub)=\\
&=(f\cdot \theta^{|u|}(e),b^{-1}ua)=\\
&=(f,b^{-1}b)*(e,a^{-1}a).
\end{align*}

3. Якщо $a\notin\operatorname{suff}(b)$ і $b\notin\operatorname{suff}(a)$, то
\begin{equation*}
(e,a^{-1}a)*(f,b^{-1}b)=\boldsymbol{0}=(f,b^{-1}b)*(e,a^{-1}a).
\end{equation*}

Очевидно, що кожен елемент напівгрупи $\mathscr{P}_{\lambda}(\theta,S)$ комутує з її нулем.
\end{proof}

Нехай $S$ -- напівгрупа. Для довільних $a_{1}^{-1}a_{2}\in P_{\lambda}$, $A\subseteq S$ і $s\in S$ позначимо
\begin{align*}
S_{a_{1}^{-1}a_{2}}&=\left\{(t,a_{1}^{-1}a_{2})\in\mathscr{P}_{\lambda}(\theta,S)\colon t\in S\right\},\\
A_{a_{1}^{-1}a_{2}}&=\left\{(t,a_{1}^{-1}a_{2})\in\mathscr{P}_{\lambda}(\theta,S)\colon t\in A\subseteq S\right\},\\
P_{\lambda}^{s}&=\left\{(s,b_{1}^{-1}b_{2})\in\mathscr{P}_{\lambda}(\theta,S)\colon b_{1}^{-1}b_{2}\in P_{\lambda}\setminus\left\{0_{P_{\lambda}}\right\}\right\}\cup\{\boldsymbol{0}\}.
\end{align*}

\begin{proposition}\label{proposition-4a}
\begin{itemize}
  \item[1.] Множина $S_{a_{1}^{-1}a_{2}}$ з індукованою з $\mathscr{P}_{\lambda}(\theta,S)$ операцією ізоморфна напівгрупі $S$ тоді і тільки тоді, коли $a_{1}=a_{2}$.
  \item[2.]  Множина $P_{\lambda}^{s}$ з індукованою з $\mathscr{P}_{\lambda}(\theta,S)$ операцією ізоморфна поліциклічному моноїду $P_{\lambda}$ тоді і тільки тоді, коли $s$ -- ідемпотент напівгрупи $S$.
\end{itemize}
\end{proposition}

\begin{proof}
1. Якщо $a_{1}\neq a_{2}$, то з тверження \ref{proposition-2} випливає, що $S_{a_{1}^{-1}a_{2}}$ не є піднапівгрупою напівгрупи $\mathscr{P}_{\lambda}(\theta,S)$.

У випадку  $a_{1}=a_{2}$ визначимо відображення $f\colon S_{a_{1}^{-1}a_{2}}\rightarrow S$ за формулою $f((s,a_{1}^{-1}a_{2}))=s$. Очевидно, що відображення $f$ є бієктивним. Доведемо, що воно зберігає операцію. Нехай $(s,a_{1}^{-1}a_{2})$ і $(t,a_{1}^{-1}a_{2})$ -- довільні елементи з $S_{a_{1}^{-1}a_{2}}$. Тоді
\begin{equation*}
f((s,a_{1}^{-1}a_{2})*(t,a_{1}^{-1}a_{2}))=f((st,a_{1}^{-1}a_{2}))=st=f((s,a_{1}^{-1}a_{2}))\cdot f((t,a_{1}^{-1}a_{2})).
\end{equation*}
Отже, $f$ є ізоморфізмом.

2. Якщо $s$ не є ідемпотентом напівгрупи $S$, то
\begin{equation*}
(s,1_{P_{\lambda}})*(s,1_{P_{\lambda}})=(ss,1_{P_{\lambda}})\notin P_{\lambda}^{s},
\end{equation*}
а отже, $P_{\lambda}^{s}$ не є піднапівгрупою напівгрупи $\mathscr{P}_{\lambda}(\theta,S)$.

У випадку, коли $s$ -- ідемпотент напівгрупи $S$, то визначимо відображення $f\colon P_{\lambda}^{s}\rightarrow P_{\lambda}$ за формулами $f((s,a_{1}^{-1}a_{2}))=a_{1}^{-1}a_{2}$ і $f(\boldsymbol{0})=0_{P_{\lambda}}$. Очевидно, що так означене відображення $f$ є бієктивним. Доведемо, що воно зберігає операцію. Нехай $(s,a_{1}^{-1}a_{2}),(s,b_{1}^{-1}b_{2})\in P_{\lambda}^{s}$. Розглянемо можливі випадки.

a) Існує слово $u\in\lambda^{*}$ таке, що $b_{1}=ua_{2}$. Тоді:
\begin{align*}
f((s,a_{1}^{-1}a_{2})*(s,b_{1}^{-1}b_{2}))&=
f((\theta^{|u|}(s)s,a_{1}^{-1}u^{-1}b_{2}))=\\
&=f((1_{S}s,a_{1}^{-1}u^{-1}b_{2}))=\\
&=f((s,a_{1}^{-1}u^{-1}b_{2}))=\\
&=a_{1}^{-1}u^{-1}b_{2}=\\
&=a_{1}^{-1}a_{2}\cdot b_{1}^{-1}b_{2}=\\
&=f((s,a_{1}^{-1}a_{2}))*f((s,b_{1}^{-1}b_{2})).
\end{align*}

b) Існує слово $u\in\lambda^{*}$ таке, що $a_{2}=ub_{1}$. Тоді:
\begin{align*}
f((s,a_{1}^{-1}a_{2})*(s,b_{1}^{-1}b_{2}))&=
f((s\theta^{|u|}(s),a_{1}^{-1}ub_{2}))=\\
&=f((s1_{S},a_{1}^{-1}ub_{2}))=\\
&=f((s,a_{1}^{-1}ub_{2}))=\\
&=a_{1}^{-1}ub_{2}=\\
&=a_{1}^{-1}a_{2}\cdot b_{1}^{-1}b_{2}=\\
&=f((s,a_{1}^{-1}a_{2}))*f((s,b_{1}^{-1}b_{2})).
\end{align*}

c) Якщо $(s,a_{1}^{-1}a_{2})*(s,b_{1}^{-1}b_{2})=\boldsymbol{0}$, то
\begin{align*}
f((s,a_{1}^{-1}a_{2})*(s,b_{1}^{-1}b_{2}))=\boldsymbol{0}=a_{1}^{-1}a_{2}\cdot b_{1}^{-1}b_{2}=f((s,a_{1}^{-1}a_{2}))*f((s,b_{1}^{-1}b_{2})).
\end{align*}
\end{proof}

\begin{proposition}\label{proposition-4}
Елемент $(t,b_{1}^{-1}b_{2})$ є інверсним до елемента $(s,a_{1}^{-1}a_{2})$ в напівгрупі $\mathscr{P}_{\lambda}(\theta,S)$ тоді і тільки тоді, коли $b_{1}=a_{2},b_{2}=a_{1}$ і $t$ --- інверсний елемент до $s$ в напівгрупі $S$.
\end{proposition}

\begin{proof}
$(\Rightarrow)$
Нехай елемент $(t,b_{1}^{-1}b_{2})$ є інверсним до елемента $(s,a_{1}^{-1}a_{2})$ в напівгрупі $\mathscr{P}_{\lambda}(\theta,S)$. Розглянемо можливі випадки.

1. Існують слова $u,v\in\lambda^{*}$ такі, що $b_{1}=ua_{2}$ і $a_{1}=vb_{2}$. Тоді
\begin{align*}
(s,a_{1}^{-1}a_{2})*(t,b_{1}^{-1}b_{2})*(s,a_{1}^{-1}a_{2})&=
(\theta^{|u|}(s)t,(ua_{1})^{-1}b_{2})*(s,a_{1}^{-1}a_{2})=\\
&=(\theta^{|v|}(\theta^{|u|}(s)t)s,(vua_{1})^{-1}a_{2})=\\
&=(s,a_{1}^{-1}a_{2}).
\end{align*}
З цієї рівності випливає, що $a_{1}=vua_{1}$. Тому $u=\varepsilon$ і $v=\varepsilon$, а отже, $b_{1}=a_{2}$ і $b_{2}=a_{1}$.

2.  Існують слова $u,v\in\lambda^{*}$ такі, що $a_{1}=ub_{2}$ і $a_{2}=vb_{1}$. Тоді
\begin{align*}
(t,b_{1}^{-1}b_{2})*(s,a_{1}^{-1}a_{2})*(t,b_{1}^{-1}b_{2})&=
(\theta^{|u|}(t)s,(ub_{1})^{-1}a_{2})*(t,b_{1}^{-1}b_{2})=\\
&=(\theta^{|u|}(t)s\theta^{|v|}(t),(ub_{1})^{-1}vb_{2})=\\
&=(t,b_{1}^{-1}b_{2}).
\end{align*}
З цієї рівності випливає, що $b_{1}=ub_{1}$ і $b_{2}=vb_{2}$. Тому $u=\varepsilon$ і $v=\varepsilon$, а отже, $b_{1}=a_{2}$ і $b_{2}=a_{1}$.

3. Існують слова $u,v\in\lambda^{*}$ такі, що $b_{1}=ua_{2}$ і $b_{2}=va_{1}$. Тоді
\begin{align*}
(s,a_{1}^{-1}a_{2})*(t,b_{1}^{-1}b_{2})*(s,a_{1}^{-1}a_{2})&=
(\theta^{|u|}(s)t,(ua_{1})^{-1}b_{2})*(s,a_{1}^{-1}a_{2})=\\
&=(\theta^{|u|}(s)t\theta^{|v|}(s),(ua_{1})^{-1}va_{2})=\\
&=(s,a_{1}^{-1}a_{2}).
\end{align*}
З цієї рівності випливає, що $a_{1}=ua_{1}$ і $a_{2}=va_{2}$. Тому $u=\varepsilon$ і $v=\varepsilon$, а отже, $b_{1}=a_{2}$ і $b_{2}=a_{1}$.

4. Існують слова $u,v\in\lambda^{*}$ такі, що $a_{2}=ub_{1}$ і $b_{2}=va_{1}$. Тоді
\begin{align*}
(s,a_{1}^{-1}a_{2})*(t,b_{1}^{-1}b_{2})*(s,a_{1}^{-1}a_{2})&=
(s\theta^{|u|}(t),a_{1}^{-1}ub_{2})*(s,a_{1}^{-1}a_{2})=\\
&=(s\theta^{|u|}(t)\theta^{|v|}(s),a_{1}^{-1}uva_{2})=\\
&=(s,a_{1}^{-1}a_{2}).
\end{align*}
З цієї рівності випливає, що $a_{2}=vua_{2}$. Тому $u=\varepsilon$ і $v=\varepsilon$, а отже, $b_{1}=a_{2}$ і $b_{2}=a_{1}$.

Отже, якщо елемент $(t,b_{1}^{-1}b_{2})$ є інверсним до елемента $(s,a_{1}^{-1}a_{2})$ в напівгрупі $\mathscr{P}_{\lambda}(\theta,S)$, то $b_{1}=a_{2}$ і $b_{2}=a_{1}$. Тому
\begin{align*}
(s,a_{1}^{-1}a_{2})*(t,b_{1}^{-1}b_{2})*(s,a_{1}^{-1}a_{2})&=
(s,a_{1}^{-1}a_{2})*(t,a_{2}^{-1}a_{1})*(s,a_{1}^{-1}a_{2})=\\
&=(sts,a_{1}^{-1}a_{2})=\\
&=(s,a_{1}^{-1}a_{2})
\end{align*}
і
\begin{align*}
(t,b_{1}^{-1}b_{2})*(s,a_{1}^{-1}a_{2})*(t,b_{1}^{-1}b_{2})&=
(t,a_{2}^{-1}a_{1})*(s,a_{1}^{-1}a_{2})*(t,a_{2}^{-1}a_{1})=\\
&=(tst,a_{2}^{-1}a_{1})=\\
&=(t,a_{2}^{-1}a_{1}).
\end{align*}
З першої рівності випливає, що $s=sts$, а з другої випливає, що $t=tst$. Отжн, елемент $t$ є інверсним до елемента $s$ у напівгрупі $S$.

$(\Leftarrow)$
Нехай $s$ -- довільний елемент напівгрупи $S$ і $s^{-1}$~--- інверсний елемент до $s$ у напівгрупі $S$. Тоді
\begin{align*}
(s,a_{1}^{-1}a_{2})*(s^{-1},a_{2}^{-1}a_{1})*(s,a_{1}^{-1}a_{2})=(ss^{-1}s,a_{1}^{-1}a_{2})=(s,a_{1}^{-1}a_{2})
\end{align*}
і
\begin{align*}
(s^{-1},a_{2}^{-1}a_{1})*(s,a_{1}^{-1}a_{2})*(s^{-1},a_{2}^{-1}a_{1})=(s^{-1}ss^{-1},a_{1}^{-1}a_{2})=(s^{-1},a_{2}^{-1}a_{1}).
\end{align*}
\end{proof}

З твердження \ref{proposition-4} випливають такі два наслідки.

\begin{corollary}\label{corollary-1}
Напівгрупа $\mathscr{P}_{\lambda}(\theta,S)$ -- регулярна тоді і тільки тоді, коли напівгрупа $S$ -- регулярна.
\end{corollary}

\begin{corollary}\label{corollary-2}
Напівгрупа $\mathscr{P}_{\lambda}(\theta,S)$ -- інверсна тоді і тільки тоді, коли напівгрупа $S$ -- інверсна.
\end{corollary}

\begin{lemma}\label{lemma-2}
Нехай $(s,a_{1}^{-1}a_{2})$ і $(t,b_{1}^{-1}b_{2})$ -- довільні ненульові елементи напівгрупи $\mathscr{P}_{\lambda}(\theta,S)$. Якщо $(s,a_{1}^{-1}a_{2})* (t,b_{1}^{-1}b_{2})=(r,c_{1}^{-1}c_{2})\neq \boldsymbol{0}$, то $a_{1}\in\operatorname{suff}(c_{1})$ і $b_{2}\in\operatorname{suff}(c_{2})$.
\end{lemma}

\begin{proof}
Розглянемо можливі випадки, коли $(r,c_{1}^{-1}c_{2})\neq\boldsymbol{0}$.

1. Існує слово $u\in\lambda^{*}$ таке, що $a_{2}=ub_{1}$. Тоді
\begin{align*}
(s,a_{1}^{-1}a_{2})*(t,b_{1}^{-1}b_{2})=(s,a_{1}^{-1}ub_{1})*(t,b_{1}^{-1}b_{2})=(s\cdot\theta^{|u|}(t),a_{1}^{-1}ub_{2}).
\end{align*}

2. Існує слово $v\in\lambda^{*}$ таке, що $b_{1}=va_{2}$. Тоді
\begin{align*}
(s,a_{1}^{-1}a_{2})*(t,b_{1}^{-1}b_{2})=(s,a_{1}^{-1}a_{2})*(t,(va_{2})^{-1}b_{2})=(\theta^{|v|}(s)\cdot t,(va_{1})^{-1}b_{2}).
\end{align*}

Отже, $a_{1}\in\operatorname{suff}(c_{1})$ і $b_{2}\in\operatorname{suff}(c_{2})$.
\end{proof}

\begin{theorem}\label{theorem-1}
Нехай $(s,a_{1}^{-1}a_{2})$ і $(t,b_{1}^{-1}b_{2})$ -- довільні ненульові елементи напівгрупи $\mathscr{P}_{\lambda}(\theta,S)$. Тоді:
\begin{itemize}
\itemsep=0pt
\parskip=1pt
  \item[1)] $(s,a_{1}^{-1}a_{2})\mathscr{L}(t,b_{1}^{-1}b_{2})$ в $\mathscr{P}_{\lambda}(\theta,S)$ тоді і тільки тоді, коли $s\mathscr{L}t$ в $S$ і $a_{2}=b_{2}$;
  \item[2)] $(s,a_{1}^{-1}a_{2})\mathscr{R}(t,b_{1}^{-1}b_{2})$ в $\mathscr{P}_{\lambda}(\theta,S)$ тоді і тільки тоді, коли $s\mathscr{R}t$ в $S$ і $a_{1}=b_{1}$;
  \item[3)] $(s,a_{1}^{-1}a_{2})\mathscr{H}(t,b_{1}^{-1}b_{2})$ в $\mathscr{P}_{\lambda}(\theta,S)$ тоді і тільки тоді, коли $s\mathscr{H}t$ в $S$ і $a_{1}^{-1}a_{2}=b_{1}^{-1}b_{2}$;
  \item[4)] $(s,a_{1}^{-1}a_{2})\mathscr{D}(t,b_{1}^{-1}b_{2})$ в $\mathscr{P}_{\lambda}(\theta,S)$ тоді і тільки тоді, коли $s\mathscr{D}t$ в $S$.
\end{itemize}
\end{theorem}

\begin{proof}

1. $(\Rightarrow)$
Нехай $(s,a_{1}^{-1}a_{2})\mathscr{L}(t,b_{1}^{-1}b_{2})$ в $\mathscr{P}_{\lambda}(\theta,S)$. Тоді існують елементи $(r,c_{1}^{-1}c_{2})$ i $(q,d_{1}^{-1}d_{2})$ напівгрупи $\mathscr{P}_{\lambda}(\theta,S)$ такі, що
\begin{equation*}
(s,a_{1}^{-1}a_{2})=(r,c_{1}^{-1}c_{2})*(t,b_{1}^{-1}b_{2}) \qquad \hbox{і} \qquad (t,b_{1}^{-1}b_{2})=(q,d_{1}^{-1}d_{2})*(s,a_{1}^{-1}a_{2}).
\end{equation*}
З першої рівності та леми \ref{lemma-2} випливає, що $b_{2}\in\operatorname{suff}(a_{2})$, а з другої рівності та леми \ref{lemma-2} випливає, що $a_{2}\in\operatorname{suff}(b_{2})$. Тому $a_{2}=b_{2}$, а це означає, що існують слова $u,v\in\lambda^{*}$ такі, що $a_{1}=ud_{2}$ і $b_{1}=vc_{2}$. Отже,
\begin{equation*}
(s,a_{1}^{-1}a_{2})=(r,c_{1}^{-1}c_{2})*(t,b_{1}^{-1}b_{2})=(\theta^{|v|}(r)t,(vc_{1})^{-1}b_{2})
\end{equation*}
і
\begin{equation*}
(t,b_{1}^{-1}b_{2})=(q,d_{1}^{-1}d_{2})*(s,a_{1}^{-1}a_{2})=(\theta^{|u|}(q)s,(ud_{1})^{-1}a_{2}).
\end{equation*}
З першої рівності випливає, що $s=\theta^{|v|}(r)t$, а з другої рівності випливає, що $t=\theta^{|u|}(q)s$. Отже, $s\mathscr{L}t$ в напівгрупі $S$.

$(\Leftarrow)$
Нехай $(s,a_{1}^{-1}a_{2})$ і $(t,b_{1}^{-1}b_{2})$ -- ненульові елементи напівгрупи $\mathscr{P}_{\lambda}(\theta,S)$ такі, що $s\mathscr{L}t$ в $S$ і $a_{2}=b_{2}$. Тоді існують елементи $r$ і $q$ напівгрупи $S$ такі, що $s=rt$ і $t=qs$. Тому
\begin{equation*}
(s,a_{1}^{-1}a_{2})=(r,a_{1}^{-1}b_{1})*(t,b_{1}^{-1}b_{2})\qquad \hbox{і} \qquad (t,b_{1}^{-1}b_{2})=(q,b_{1}^{-1}a_{1})*(s,a_{1}^{-1}a_{2}).
\end{equation*}
Отже, $(s,a_{1}^{-1}a_{2})\mathscr{L}(t,b_{1}^{-1}b_{2})$ в напівгрупі $\mathscr{P}_{\lambda}(\theta,S)$.

2. $(\Rightarrow)$
Нехай $(s,a_{1}^{-1}a_{2})\mathscr{R}(t,b_{1}^{-1}b_{2})$ в  $\mathscr{P}_{\lambda}(\theta,S)$. Тоді існують елементи $(r,c_{1}^{-1}c_{2})$ i $(q,d_{1}^{-1}d_{2})$ напівгрупи $\mathscr{P}_{\lambda}(\theta,S)$ такі, що
\begin{equation*}
(s,a_{1}^{-1}a_{2})=(t,b_{1}^{-1}b_{2})*(r,c_{1}^{-1}c_{2})\qquad \hbox{і} \qquad (t,b_{1}^{-1}b_{2})=(s,a_{1}^{-1}a_{2})*(q,d_{1}^{-1}d_{2}).
\end{equation*}
З першої рівності та леми \ref{lemma-2} випливає, що $b_{1}\in\operatorname{suff}(a_{1})$, а з другої рівності та леми~\ref{lemma-2} випливає, що $a_{1}\in\operatorname{suff}(b_{1})$. Тому $a_{1}=b_{1}$, а це означає, що існують слова $u,v\in\lambda^{*}$ такі, що $a_{2}=ud_{1}$ і $b_{2}=vc_{1}$. Отож, отримуємо, що
\begin{equation*}
(s,a_{1}^{-1}a_{2})=(t,b_{1}^{-1}b_{2})*(r,c_{1}^{-1}c_{2})=(t\theta^{|v|}(r),b_{1}^{-1}vc_{2})
\end{equation*}
і
\begin{equation*}
(t,b_{1}^{-1}b_{2})=(s,a_{1}^{-1}a_{2})*(q,d_{1}^{-1}d_{2})=(s\theta^{|u|}(q),a_{1}^{-1}ud_{2}).
\end{equation*}
З першої рівності випливає, що $s=t\theta^{|v|}(r)$, а з другої рівності випливає, що $t=s\theta^{|u|}(q)$. Отже, $s\mathscr{R}t$ в напівгрупі $S$.

$(\Leftarrow)$
Нехай $(s,a_{1}^{-1}a_{2})$ і $(t,b_{1}^{-1}b_{2})$ -- ненульові елементи напівгрупи $\mathscr{P}_{\lambda}(\theta,S)$ такі, що $s\mathscr{L}t$ в $S$ і $a_{1}=b_{1}$. Тоді існують елементи $r$ і $q$ напівгрупи $S$ такі, що $s=tr$ і $t=sq$. Тому
\begin{equation*}
(s,a_{1}^{-1}a_{2})=(t,b_{1}^{-1}b_{2})*(r,b_{2}^{-1}b_{2})\qquad \hbox{і} \qquad (t,b_{1}^{-1}b_{2})=(s,a_{1}^{-1}a_{2})*(q,a_{2}^{-1}a_{2}).
\end{equation*}
Отже, $(s,a_{1}^{-1}a_{2})\mathscr{R}(t,b_{1}^{-1}b_{2})$ в напівгрупі $\mathscr{P}_{\lambda}(\theta,S)$.

3. Випливає з тверджень 1 і 2.

4. $(\Rightarrow)$
Нехай $(s,a_{1}^{-1}a_{2})\mathscr{D}(t,b_{1}^{-1}b_{2})$ у напівгрупі $\mathscr{P}_{\lambda}(\theta,S)$. Тоді існує елемент $(r,c_{1}^{-1}c_{2})\in \mathscr{P}_{\lambda}(\theta,S)$ такий, що  $(s,a_{1}^{-1}a_{2})\mathscr{L}(r,c_{1}^{-1}c_{2})$ і $(r,c_{1}^{-1}c_{2})\mathscr{L}(t,b_{1}^{-1}b_{2})$. За тверд\-женнями 1 і 2 отримаємо, що  $s\mathscr{L}r$ і $r\mathscr{R}t$ в $S$, а отже, $s\mathscr{D}t$ у напівгрупі $S$.

$(\Leftarrow)$
Нехай $(s,a_{1}^{-1}a_{2})$ і $(t,b_{1}^{-1}b_{2})$ -- ненульові елементи напівгрупи $\mathscr{P}_{\lambda}(\theta,S)$ такі, що $s\mathscr{D}t$ в $S$. Тоді існує елемент $r$ напівгрупи $S$ такий, що $s\mathscr{L}r$ і $r\mathscr{R}t$. Тому $(s,a_{1}^{-1}a_{2})\mathscr{L}(r,b_{1}^{-1}a_{2})$ і $(r,b_{1}^{-1}a_{2})\mathscr{R}(t,b_{1}^{-1}b_{2})$. Отже, $(s,a_{1}^{-1}a_{2})\mathscr{D}(t,b_{1}^{-1}b_{2})$ у напівгрупі $\mathscr{P}_{\lambda}(\theta,S)$.
\end{proof}

З теореми \ref{theorem-1} і наслідку \ref{corollary-2} випливають такі два наслідки.

\begin{corollary}\label{corollary-3}
Напівгрупа $\mathscr{P}_{\lambda}(\theta,S)$ -- комбінаторна тоді і тільки тоді, коли напівгрупа $S$ -- комбінаторна.
\end{corollary}

\begin{corollary}\label{corollary-5}
Напівгрупа $\mathscr{P}_{\lambda}(\theta,S)$ -- 0-біпроста тоді і тільки тоді, коли напівгрупа $S$ -- бі\-проста.
\end{corollary}





\begin{proposition}\label{proposition-5}
Напівгрупа $\mathscr{P}_{\lambda}(\theta,S)$ є 0-простою для довільної напівгрупи $S$.
\end{proposition}

\begin{proof}
Нехай $(s,a_{1}^{-1}a_{2})$ і $(t,b_{1}^{-1}b_{2})$ -- довільні ненульові елементи напівгрупи $\mathscr{P}_{\lambda}(\theta,S)$ і $u$ -- непорожнє слово вільного моноїда $\lambda^{*}$. Тоді
\begin{align*}
((\theta^{|u|}(t))^{-1},a_{1}^{-1}ub_{1})&*(t,b_{1}^{-1}b_{2})*(s,(ub_{2})^{-1}a_{2})=\\
&=((\theta^{|u|}(t))^{-1}\cdot\theta^{|u|}(t),a_{1}^{-1}ub_{2})*(s,(ub_{2})^{-1}a_{2})=\\
&=(1_{S},a_{1}^{-1}ub_{2})*(s,(ub_{2})^{-1}a_{2})=\\
&=(s,a_{1}^{-1}a_{2}),
\end{align*}
звідки випливає, що напівгрупа $\mathscr{P}_{\lambda}(\theta,S)$ є 0-простою.
\end{proof}

\begin{proposition}\label{proposition-6}
Напівгрупа $\mathscr{P}_{\lambda}(\theta,S)$ --- конгруенц-проста тоді і тільки тоді, коли $S$ -- тривіальна напівгрупа.
\end{proposition}

\begin{proof}
Якщо $S$ -- одноелементна множина, то напівгрупа $\mathscr{P}_{\lambda}(\theta,S)$ ізоморфна $\lambda$-полі\-цик\-ліч\-но\-му моноїду $P_{\lambda}$. Оскільки $\lambda$-поліциклічний моноїд є конгруенц-простою напівгрупою (див. на\-прик\-лад \cite[теорема 2.5]{Bardyla-Gutik-2016}), то у цьому випадку напівгрупа $\mathscr{P}_{\lambda}(\theta,S)$ є конгруенц-простою.

Нехай напівгрупа $S$ не є одноелементною множиною. Тоді відношення
\begin{align*}
\mathfrak{C}=\left\{((x,a_{1}^{-1}a_{2}),(y,a_{1}^{-1}a_{2}))\colon x,y\in S,a_{1}^{-1}a_{2}\in P_{\lambda}\right\}\cup \{(\boldsymbol{0},\boldsymbol{0})\}
\end{align*}
є нетривіальною конгруенцією на $\mathscr{P}_{\lambda}(\theta,S)$, причому, очевидно, що фактор-на\-пів\-група $\mathscr{P}_{\lambda}(\theta,S)/\mathfrak{C}$ ізоморфна $\lambda$-полі\-цик\-ліч\-но\-му моноїду $P_{\lambda}$.
\end{proof}

\begin{proposition}\label{proposition-7}
Інверсна напівгрупа $\mathscr{P}_{\lambda}(\theta,S)$ є 0-E-унітарною тоді і тільки тоді, коли напівгрупа $S$ є інверсною E-унітарною та $\theta^{-1}(1_{S})=E(S)$.
\end{proposition}

\begin{proof}
$(\Rightarrow)$
Якщо напівгрупа $\mathscr{P}_{\lambda}(\theta,S)$ -- інверсна, то з наслідку \ref{corollary-2} випливає, що напівгрупа $S$ є також інверсною. Нехай $s$ -- елемент напівгрупи $S$ і $e$ -- ідемпотенти напівгрупи $S$ такі, що $es\in E(S)$. Тоді $(e,1_{P_{\lambda}})*(s,1_{P_{\lambda}})=(es,1_{P_{\lambda}})$. Оскільки $(e,1_{P_{\lambda}})$ і $(ex,1_{P_{\lambda}})$ -- ідемпотенти напівгрупи $\mathscr{P}_{\lambda}(\theta,S)$ та інверсна напівгрупа $\mathscr{P}_{\lambda}(\theta,S)$ є 0-E-унітарною, то $(x,1_{P_{\lambda}})\in E(\mathscr{P}_{\lambda}(\theta,S))$, а отже, $s$ є ідемпотентом напівгрупи $S$. Тому інверсна напівгрупа $S$ є E-унітарною.

Нехай $s$ -- елемент напівгрупи $S$ такий, що $\theta (s)=1_{S}$ і $a$ -- слово в $\lambda^{*}$ таке, що $|a|=1$. Тоді
\begin{equation*}
(1_{S},a^{-1}a)*(s,1_{P_{\lambda}})=(\theta(s),a^{-1}a)=(1_{S},a^{-1}a).
\end{equation*}
Оскільки $(1_{S},a^{-1}a)$ -- ідемпотент напівгрупи $\mathscr{P}_{\lambda}(\theta,S)$ й інверсна напівгрупа  $\mathscr{P}_{\lambda}(\theta,S)$ є 0-E-уні\-тар\-ною, то елемент $(s,a^{-1}a)$ міститься в  $E(\mathscr{P}_{\lambda}(\theta,S))$, а отже, $s$ -- ідемпотент напівгрупи $S$.

$(\Leftarrow)$
Нехай $(e,a^{-1}a)$ -- ненульовий ідемпотент напівгрупи $\mathscr{P}_{\lambda}(\theta,S)$ і $(s,b_{1}^{-1}b_{2})$ -- елемент напівгрупи $\mathscr{P}_{\lambda}(\theta,S)$ такі, що
\begin{equation*}
(e,a^{-1}a)* (s,b_{1}^{-1}b_{2})\in E(\mathscr{P}_{\lambda}(\theta,S))\setminus\{\boldsymbol{0}\}.
\end{equation*}
Розглянемо можливі випадки.

1. Якщо існує слово $u\in\lambda^{*}$ таке, що $b_{1}=ua$, то
\begin{equation*}
(e,a^{-1}a)*(s,b_{1}^{-1}b_{2})=(\theta^{|u|}(e)\cdot s,(ua)^{-1}b_{2})=(s,b_{1}^{-1}b_{2}).
\end{equation*}
Оскільки $(s,b_{1}^{-1}b_{2})$ є ідемпотентом напівгрупи $\mathscr{P}_{\lambda}(\theta,S)$, то з твердження \ref{proposition-2} випливає, що $b_{1}=b_{2}$ і $s$ є ідемпотентом напівгрупи $S$. Отже, $(s,b_{1}^{-1}b_{2})\in E(\mathscr{P}_{\lambda}(\theta,S))$.

2. Якщо існує слово $v\in\lambda^{*}$ таке, що $a=vb_{1}$, то
\begin{equation*}
(e,a^{-1}a)*(s,b_{1}^{-1}b_{2})=(e\cdot\theta^{|v|}(s),a^{-1}vb_{2}).
\end{equation*}
Оскільки $(e\cdot\theta^{|v|}(s),a^{-1}vb_{2})$ є ідемпотентом напівгрупи $\mathscr{P}_{\lambda}(\theta,S)$, то $vb_{1}=a=vb_{2}$ і $\theta^{|v|}(s)$ є ідемпотентом напівгрупи $S$. Враховуючи попереднє і те, що $\theta^{-1}(1_{S})=E(S)$ отримуємо, що $b_{1}=b_{2}$ і $s\in E(S)$. Отже, $(s,b_{1}^{-1}b_{2})$ є ідемпотентом напівгрупи $\mathscr{P}_{\lambda}(\theta,S)$.

Тому інверсна напівгрупа $\mathscr{P}_{\lambda}(\theta,S)$ -- 0-E-унітарна.
\end{proof}

\begin{proposition}\label{proposition-8}
Ненульовий елемент $(s,a_{1}^{-1}a_{2})$ напівгрупи $\mathscr{P}_{\lambda}(\theta,S)$ належить центру $Z(\mathscr{P}_{\lambda}(\theta,S))$ тоді і тільки тоді, коли $s\in Z(S)$, $s=\theta(s)$ і $a_{1}^{-1}a_{2}=1_{P_{\lambda}}$.
\end{proposition}

\begin{proof}
$(\Rightarrow)$
Нехай $(s,a_{1}^{-1}a_{2})$ -- ненульовий елемент з центру напівгрупи $\mathscr{P}_{\lambda}(\theta,S)$. Тоді
\begin{align*}
(s,a_{1}^{-1}a_{2}a_{1})=(s,a_{1}^{-1}a_{2})*(1_{S},a_{1})=(1_{S},a_{1})*(s,a_{1}^{-1}a_{2})=(s,a_{2})
\end{align*}
і
\begin{align*}
(s,a_{1}^{-1})=(s,a_{1}^{-1}a_{2})*(1_{S},a_{2}^{-1})=(1_{S},a_{2}^{-1})*(s,a_{1}^{-1}a_{2})=(s,(a_{1}a_{2})^{-1}a_{2}).
\end{align*}
З першої рівності випливає, що $a_{1}=\varepsilon$, а з другої випливає, що $b=\varepsilon$. Тому $a_{1}^{-1}a_{2}=1_{P_{\lambda}}$.

Нехай $(s,1_{P_{\lambda}})$ -- ненульовий елемент з центру напівгрупи $\mathscr{P}_{\lambda}(\theta,S)$ і $t$ -- довільний елемент напівгрупи $S$. Тоді
\begin{align*}
(s\cdot t,1_{P_{\lambda}})=(s,1_{P_{\lambda}})*(t,1_{P_{\lambda}})=(t,1_{P_{\lambda}})*(s,1_{P_{\lambda}})=(t\cdot s,1_{P_{\lambda}}).
\end{align*}
Тому $s\cdot t=t\cdot s$, а отже, $s$ належить центру напівгрупи $S$.

Нехай $(s,1_{P_{\lambda}})$ -- ненульовий елемент з центру напівгрупи $\mathscr{P}_{\lambda}(\theta,S)$ і $a$ -- слово з $\lambda^{*}$ таке, що $|a|=1$. Тоді
\begin{align*}
(s,a)=(s,1_{P_{\lambda}})*(1_{S},a)=(1_{S},a)*(s,1_{P_{\lambda}})=(\theta(s),a).
\end{align*}
Звідси випливає, що $s=\theta(s)$.

$(\Leftarrow)$
Нехай $(t,a_{1}^{-1}a_{2})$ -- довільний елемент напівгрупи $\mathscr{P}_{\lambda}(\theta,S)$ і $s$ -- елемент з центру напівгрупи $S$ такий, що $s=\theta(s)$. Тоді
\begin{align*}
(s,1_{P_{\lambda}})*(t,a_{1}^{-1}a_{2})&=(\theta^{|a_{1}|}(s)\cdot t,a_{1}^{-1}a_{2})=\\
&=(s\cdot t,a_{1}^{-1}a_{2})=\\
&=(t\cdot s,a_{1}^{-1}a_{2})=\\
&=(t\cdot \theta^{|a_{2}|}(s),a_{1}^{-1}a_{2})=\\
&=(t,a_{1}^{-1}a_{2})*(s,1_{P_{\lambda}}).
\end{align*}
Тому $(s,1_{P_{\lambda}})$ належить центру напівгрупи $\mathscr{P}_{\lambda}(\theta,S)$.
\end{proof}

\begin{proposition}\label{proposition-9}
Елемент $(s,a_{1}^{-1}a_{2})$ з напівгрупи $\mathscr{P}_{\lambda}(\theta,S)$ належить групі одиниць $H(1_{\mathscr{P}_{\lambda}(\theta,S)})$ тоді і тільки тоді, коли $s\in H(1_{S})$ і $a_{1}^{-1}a_{2}=1_{P_{\lambda}}$.
\end{proposition}
\begin{proof}

$(\Rightarrow)$
З леми \ref{lemma-2} випливає таке: якщо елемент $(s,a_{1}^{-1}a_{2})$ належить групі одиниць $H(1_{\mathscr{P}_{\lambda}(\theta,S)})$ напівгрупи $\mathscr{P}_{\lambda}(\theta,S)$, то $a_{1}^{-1}a_{2}=1_{P_{\lambda}}$.

Нехай елемент $(s,1_{P_{\lambda}})$ належить групі одиниць $H(1_{\mathscr{P}_{\lambda}(\theta,S)})$. Тоді існує елемент $(t,1_{P_{\lambda}})$ групи одиниць напівгрупи $\mathscr{P}_{\lambda}(\theta,S)$ такий, що
\begin{align*}
(st,1_{P_{\lambda}})=(s,1_{P_{\lambda}})*(t,1_{P_{\lambda}})=(t,1_{P_{\lambda}})*(s,1_{P_{\lambda}})=(ts,1_{P_{\lambda}})=(1_{S},1_{P_{\lambda}}).
\end{align*}
Тому $s$ є елементом групи одиниць напівгрупи $S$.

$(\Leftarrow)$
Нехай $s$ -- елемент групи одиниць напівгрупи $S$. Тоді
\begin{align*}
(s,1_{P_{\lambda}})*(s^{-1},1_{P_{\lambda}})&=(s\cdot s^{-1},1_{P_{\lambda}})=\\
&=(s^{-1}\cdot s,1_{P_{\lambda}})=\\
&=(s^{-1},1_{P_{\lambda}})*(s,1_{P_{\lambda}})=\\
&=(1_{S},1_{P_{\lambda}}).
\end{align*}
Отже, $(s,1_{P_{\lambda}})$ є елементом групи одиниць напівгрупи $\mathscr{P}_{\lambda}(\theta,S)$.
\end{proof}

\begin{proposition}\label{proposition-11}
Множини $\left\{x\in\mathscr{P}_{\lambda}(\theta,S)\colon a*x=b\right\}$ і $\left\{x\in\mathscr{P}_{\lambda}(\theta,S)\colon x*a=b\right\}$ -- скінченні для будь-яких ненульових елементів $a,b\in\mathscr{P}_{\lambda}(\theta,S)$ тоді і тільки тоді, коли множини $\left\{x\in S\colon sx=t\right\}$, $\left\{x\in S\colon xs=t\right\}$ і $\theta^{-1}(s)$ -- скінченні для будь-яких $s,t\in S$.
\end{proposition}

\begin{proof}
$(\Rightarrow)$
Якщо множина $\left\{x\in S\colon sx=t\right\}$ -- нескінченна для деяких $s,t\in S$, то множина $\left\{x\in\mathscr{P}_{\lambda}(\theta,S)\colon (s,1_{P_{\lambda}})*x=(t,1_{P_{\lambda}})\right\}$ -- нескінченна. Якщо ж  множина $\left\{x\in S\colon xs=t\right\}$ -- нескінченна для деяких $s,t\in S$, то множина
\begin{equation*}
\left\{x\in\mathscr{P}_{\lambda}(\theta,S)\colon x*(s,1_{P_{\lambda}})=(t,1_{P_{\lambda}})\right\}
\end{equation*}
-- нескінченна. Отже, множини $\left\{x\in S\colon sx=t\right\}$ і $\left\{x\in S\colon xs=t\right\}$ -- скінченні для будь-яких $s,t\in S$.

Припустимо, що множина $\theta^{-1}(t)$ -- нескінченна для деякого елемента $t\in S$. Тоді множина
\begin{equation*}
\left\{(x,1_{P_{\lambda}})\in\mathscr{P}_{\lambda}(\theta,S)\colon (1_{S},a)*(x,1_{P_{\lambda}})=(t,a)\right\}
\end{equation*}
є нескінченною для довільного слова $a\in\lambda^{*}$.

$(\Leftarrow)$
Нехай $(s,a_{1}^{-1}a_{2}),(t,b_{1}^{-1}b_{2})$ -- довільні елементи напівгрупи $\mathscr{P}_{\lambda}(\theta,S)$. Доведемо, що множина
\begin{equation*}
\left\{(x,y_{1}^{-1}y_{2})\in\mathscr{P}_{\lambda}(\theta,S)\colon (s,a_{1}^{-1}a_{2})*(x,y_{1}^{-1}y_{2})=(t,b_{1}^{-1}b_{2})\right\}
\end{equation*}
є  скінченною.
З твердження 2.7 \cite{Bardyla-Gutik-2016} випливає, що існує скінченна кількість елементів $y_{1}^{-1}y_{2}$ $\lambda$-поліциклічного моноїда $P_{\lambda}$ таких, що $a_{1}^{-1}a_{2}\cdot y_{1}^{-1}y_{2}=b_{1}^{-1}b_{2}$. Нехай $c_{1}^{-1}c_{2}$ -- один з цих елементів. Розглянемо можливі випадки.

1. Існує слово $u\in\lambda^{*}$ таке, що $b_{1}=ua_{2}$. Тоді
\begin{align*}
(s,a_{1}^{-1}a_{2})*(x,y_{1}^{-1}y_{2})=(\theta^{|u|}(s)x,a_{1}^{-1}u^{-1}y_{2})=(t,b_{1}^{-1}b_{2}).
\end{align*}
Оскільки множини $\{x\in S\colon  sx=t\}$ і $\{x\in S\colon xs=t\}$ -- скінченні для будь-яких $s,t\in S$, то існує скінченна кількість елементів $x$ напівгрупи $S$ таких, що $\theta^{|u|}(s)x=t$.

2. Існує слово $u\in\lambda^{*}$ таке, що $a_{2}=ub_{1}$. Тоді
\begin{align*}
(s,a_{1}^{-1}a_{2})*(x,y_{1}^{-1}y_{2})=(s\theta^{|u|}(x),a_{1}^{-1}uy_{2})=(t,b_{1}^{-1}b_{2}).
\end{align*}
Оскільки множини $\{x\in S\colon sx=t\}$, $\{x\in S\colon xs=t\}$ і $\theta^{-1}(s)$ -- скінченні для будь-яких $s,t\in S$, то існує скінченна кількість елементів $x$ напівгрупи $S$ таких, що $s\theta^{|u|}(x)=t$.

Позаяк скінченне об'єднання скінченних множин є скінченною множиною, то множина
\begin{equation*}
\left\{(x,y_{1}^{-1}y_{2})\in\mathscr{P}_{\lambda}(\theta,S)\colon (s,a_{1}^{-1}a_{2})*(x,y_{1}^{-1}y_{2})=(t,b_{1}^{-1}b_{2})\right\}
\end{equation*}
скінченна. Аналогічно доводиться, що множина
\begin{equation*}
\left\{(x,y_{1}^{-1}y_{2})\in\mathscr{P}_{\lambda}(\theta,S)\colon (x,y_{1}^{-1}y_{2})*(s,a_{1}^{-1}a_{2})=(t,b_{1}^{-1}b_{2})\right\}
\end{equation*}
скінченна.
\end{proof}

З означення напівгрупової операції на напівгрупі $\mathscr{P}_{\lambda}(\theta,S)$ і твердження 2.7 \cite{Bardyla-Gutik-2016} випливає такий наслідок

\begin{corollary}\label{corollary-12}
Для довільних $a,a_1,b,b_1\in \lambda^*$, $s,t\in S$ кожна з множин
\begin{equation*}
  A=\big\{(t,u^{-1}v)\in \mathscr{P}_{\lambda}(\theta,S)\colon (s,a^{-1}b)*(t,u^{-1}v)\in S_{a_1^{-1}b_1}\big\},
\end{equation*}
чи
\begin{equation*}
  B=\big\{(t,u^{-1}v)\in \mathscr{P}_{\lambda}(\theta,S)\colon (t,u^{-1}v)*(s,a^{-1}b)\in S_{a_1^{-1}b_1}\big\}
\end{equation*}
перетинає не більше, ніж скінченну кількість підмножин вигляду $S_{c^{-1}d}$ напівгрупи $\mathscr{P}_{\lambda}(\theta,S)$.
\end{corollary}

\begin{proposition}\label{proposition-20}
Нехай $S,T$ -- довільні напівгрупи та $\mathfrak{C}_{1}$, $\mathfrak{C}_{2}$ -- конгруенції на напівгрупах $S$ і $T$, відповідно. Якщо $f\colon S\rightarrow T$ -- ізоморфізм такий, що $s\mathfrak{C}_{1}t$ тоді і тільки тоді, коли $f(s)\mathfrak{C}_{1}f(t)$ для довільних $s,t\in S$, то відображення $\overline{f}\colon S/\mathfrak{C}_{1}\rightarrow T/\mathfrak{C}_{2}$, означене $\overline{f}([s]_{\mathfrak{C}_{1}})=[f(s)]_{\mathfrak{C}_{2}}$, є ізоморфізмом.
\end{proposition}

\begin{proof}
Очевидно, що відображення $\overline{f}$ визначено коректно та є бієктивним. Доведемо, що $\overline{f}$ зберігає операцію. Нехай $[a]_{\mathfrak{C}_{1}},[b]_{\mathfrak{C}_{1}}\in S/\mathfrak{C}_{1}$. Оскільки $\mathfrak{C}_{1}$ -- конгруенція на напівгрупі $S$, то
\begin{align*}
\overline{f}([a]_{\mathfrak{C}_{1}}\cdot[b]_{\mathfrak{C}_{1}})=\overline{f}([ab]_{\mathfrak{C}_{1}})= [f(ab)]_{\mathfrak{C}_{2}}=[f(a)\cdot f(b)]_{\mathfrak{C}_{2}}=[f(a)]_{\mathfrak{C}_{2}}\cdot[f(b)]_{\mathfrak{C}_{2}},
\end{align*}
звідки випливає наше твердження.
\end{proof}

\begin{proposition}\label{proposition-19}
Нехай $S$ i $T$~--- моноїди та $\theta\colon S\to H_S(1_T)$ i $\phi\colon T\to H_T(1_T)$~--- гомоморфізми.
Якщо $f\colon\mathscr{P}_{\lambda}(\theta,S)\rightarrow\mathscr{P}_{\lambda}(\phi,T)$ -- ізоморфізм, то напівгрупа $S$ ізоморфна напівгрупі $T$ та існує  автоморфізм $f_{p}$ поліциклічного моноїда $P_{\lambda}$ такий, що $f(S_{a_{1}^{-1}a_{2}})=T_{f_{p}(a_{1}^{-1}a_{2})}$.
\end{proposition}

\begin{proof}
Оскільки
\begin{align*}
S_{1_{P_{\lambda}}}=\left\{a\in\mathscr{P}_{\lambda}(\theta,S)\colon a*b\neq\boldsymbol{0}\hbox{ і }a*b\neq\boldsymbol{0} \hbox{ для довільного }b\in\mathscr{P}_{\lambda}(\theta,S)\right\}
\end{align*}
і
\begin{align*}
T_{1_{P_{\lambda}}}=\left\{a\in\mathscr{P}_{\lambda}(\phi,T)\colon a*b\neq\boldsymbol{0}\hbox{ і }a*b\neq\boldsymbol{0}\hbox{ для довільного }b\in\mathscr{P}_{\lambda}(\phi,T)\right\},
\end{align*}
то $f(S_{1_{P_{\lambda}}})=T_{1_{P_{\lambda}}}$. Отже, за твердженням \ref{proposition-4a} напівгрупа $S$ ізоморфна напівгрупі $T$.

Визначимо конгруенцію $\mathfrak{C}_{1}$ на напівгрупі $\mathscr{P}_{\lambda}(\theta,S)$ так: $(s,a_{1}^{-1}a_{2})\mathfrak{C}_{1}(t,b_{1}^{-1}b_{2})$, якщо $a_{1}^{-1}a_{2}=b_{1}^{-1}b_{2}$ і $\boldsymbol{0}\mathfrak{C}_{1}\boldsymbol{0}$, та аналогічно: $(s,a_{1}^{-1}a_{2})\mathfrak{C}_{1}(t,b_{1}^{-1}b_{2})$, якщо $a_{1}^{-1}a_{2}=b_{1}^{-1}b_{2}$ і $\boldsymbol{0}\mathfrak{C}_{1}\boldsymbol{0}$ -- конгруенцію $\mathfrak{C}_{2}$ на напівгрупі $\mathscr{P}_{\lambda}(\phi,T)$. Доведемо, що $(s,a_{1}^{-1}a_{2})\mathfrak{C}_{1}(t,b_{1}^{-1}b_{2})$ тоді і тільки тоді, коли $f((s,a_{1}^{-1}a_{2}))\mathfrak{C}_{2}f((t,b_{1}^{-1}b_{2}))$.
Зафіксуємо  $(s,a^{-1}),(s,a)\in\mathscr{P}_{\lambda}(\theta,S)$. Тоді $(1_{S},a)*(s,a^{-1})=(s,1_{P_{\lambda}})$ і $(s,a)*(1_{S},a^{-1})=(s,1_{P_{\lambda}})$. Оскільки $f((s,1_{P_{\lambda}}))=(t,1_{P_{\lambda}})$ для деякого елемента $t\in T$, то $f((s,a^{-1}))=(r,b^{-1})$ і $f((s,a))=(q,b)$ для деяких $r,q\in T$ і $b\in\lambda^{*}$. Зауважимо, що слово $b$ не залежить від вибору елемента $s$. Тому $f(s,a^{-1})\in S_{b^{-1}}$ і $f(s,a)\in S_{b}$ для довільного елемента $s\in S$.

Нехай $(s,a_{1}^{-1}a_{2}),(t,a_{1}^{-1}a_{2})$ -- довільні елементи напівгрупи $\mathscr{P}_{\lambda}(\theta,S)$. З тео\-реми \ref{theorem-1} випливає, що $(s,a_{1}^{-1})\mathscr{R}(s,a_{1}^{-1})$, $(s,a_{2})\mathscr{L}(s,a_{1}^{-1}a_{2})$, $(t,a_{1}^{-1})\mathscr{R}(t,a_{1}^{-1}a_{2})$ і $(t,a_{2})\mathscr{L}(t,a_{1}^{-1}a_{2})$ в $\mathscr{P}_{\lambda}(\theta,S)$. З вище доведеного та того, що ізоморфізм зберігає $\mathscr{R}$-  і $\mathscr{L}$-класи випливає, що $(s,a_{1}^{-1}a_{2})\mathfrak{C}_{1}(t,b_{1}^{-1}b_{2})$ тоді і тільки тоді, коли $f((s,a_{1}^{-1}a_{2}))\mathfrak{C}_{2}f((t,b_{1}^{-1}b_{2}))$. Отже, за твердженням \ref{proposition-20} ізоморфізм $f$ породжує ізоморфізм $\overline{f}\colon \mathscr{P}_{\lambda}(\theta,S)/\mathfrak{C}_{1}\rightarrow \mathscr{P}_{\lambda}(\phi,T)/\mathfrak{C}_{1}$. Оскільки кожна з напівгруп $\mathscr{P}_{\lambda}(\theta,S)/\mathfrak{C}_{1}$ i $\mathscr{P}_{\lambda}(\phi,T)/\mathfrak{C}_{1}$  ізоморфна $\lambda$-поліциклічному моноїдові, то існує автоморфізм $f_{p}$ поліциклічного моноїда $P_{\lambda}$ такий, що $\overline{f}([(s,a_{1}^{-1}a_{2})]_{\mathfrak{C}_{1}})=\big[(s,f_{p}(a_{1}^{-1}a_{2}))\big]_{\mathfrak{C}_{2}}$.
\end{proof}

Якщо $f\colon\mathscr{P}_{\lambda}(\theta,S)\rightarrow\mathscr{P}_{\lambda}(\phi,T)$ -- ізоморфізм, то  через $f_{S}^{T}$ позначимо ізоморфізм між напівгрупами $S$ і $T$, який породжується ізоморфізмом $f|_{S_{1_{P_{\lambda}}}}$, що є звуженням ізоморфізму $f$ на підмоноїд $S_{1_{P_{\lambda}}}$, який ізоморфний поліциклічному моноїдові $P_{\lambda}$.

З твердження~\ref{proposition-19} випливає такий наслідок:

\begin{corollary}
Нехай групи одиниць напівгруп $S$ і $T$ є тривіальними. Тоді напівгрупи $\mathscr{P}_{\lambda}(\theta,S)$ і $\mathscr{P}_{\lambda}(\phi,T)$ є ізоморфними тоді і тільки тоді, коли напівгрупа $S$ ізоморфна напівгрупі $T$.
\end{corollary}

\begin{theorem}\label{theorem-3}
Нехай групи одиниць напівгруп $S$ і $T$ є тривіальними. Якщо $f\colon\mathscr{P}_{\lambda}(\theta,S)\rightarrow\mathscr{P}_{\lambda}(\phi,T)$~-- ізоморфізм, то  $f((s,a_{1}^{-1}a_{2}))=(f_{S}^{T}(s),f_{p}(a_{1}^{-1}a_{2}))$.
\end{theorem}

\begin{proof}
Нехай $f((1_{S},a))=(s,b)$ і $f((1_{S},a^{-1}))=(t,b^{-1})$. Тоді з рівностей
\begin{equation*}
(1_{S},a)*(1_{S},a^{-1})=(1_{S},1_{P_{\lambda}})\qquad \hbox{i} \qquad (1_{S},a^{-1})*(1_{S},a)=(1_{S},a^{-1}a)
\end{equation*}
і твердження \ref{proposition-4a} випливає, що $s,t\in H(1_{T})$, а отже, $s=t=1_{T}$. Розглянемо рівність
\begin{equation*}
(1_{S},a_{1})*(s,a_{1}^{-1}a_{2})*(1_{S},a_{2}^{-1})=(s,1_{P_{\lambda}}).
\end{equation*}
Оскільки $f((1_{S},a_{1}))=(1_{T},f_{p}(a_{1}))$, $f((1_{S},a_{2}^{-1}))=(1_{T},f_{p}(a_{2}^{-1}))$, $f((s,a_{1}^{-1}a_{2}))=(t,f_{p}(a_{1}^{-1}a_{2}))$  і  $f((s,1_{P_{\lambda}}))=(f_{S}^{T}(s),1_{P_{\lambda}})$, то $t=f_{S}^{T}(s)$.
\end{proof}

\section{\textbf{Топологізація напівгрупи $\mathscr{P}_{\lambda}(\theta,S)$}}

\begin{proposition}\label{proposition-3.1}
Якщо $(S,\tau_S)$ -- гаусдорфова напівтопологічна напівгрупа та $\theta\colon S\rightarrow H(1_{S})$ -- неперервний гомоморфізм, то $(\mathscr{P}_{\lambda}(\theta,S),\tau_{\texttt{s\!t}})$ -- гаусдорфова напівтопологічна напівгрупа, де $\tau_{\texttt{s\!t}}$ -- топологія породжена базою
\begin{align*}
\mathcal{B}_{\texttt{s\!t}}=\left\{U_{a^{-1}b}\colon U\in\mathcal{B}, \; a,b\in\lambda^* \right\}\cup\{\boldsymbol{0}\}
\end{align*}
і $\mathcal{B}$ -- база топології $\tau_S$.
\end{proposition}

\begin{proof}
Сім'я підмножин $\mathcal{B}_{\texttt{s\!t}}$ задовольняє умови (B1)--(B2) \cite{Engelking-1989}, а отже, вона є базою топології $\tau_{\texttt{s\!t}}$ на напівгрупі $\mathscr{P}_{\lambda}(\theta,S)$. Очевидно, що $(\mathscr{P}_{\lambda}(\theta,S),\tau_{\texttt{s\!t}})$ -- гаусдорфовий простір.

Доведемо, що $(\mathscr{P}_{\lambda}(\theta,S),\tau_{\texttt{s\!t}})$ -- напівтопологічна напівгрупа. Нехай $s,t$ -- довільні елементи напівгрупи $S$. З нарізної неперервності операції на $(S,\tau_S)$ випливає, що для довільного відкритого околу $U(st)$ елемента $st$ в $(S,\tau_S)$ існують відкриті околи $U_{1}(s),U_{2}(t)$ елементів $s,t$ в $(S,\tau_S)$. відповідно, такі, що $U_{1}(s)\cdot t\subseteq U(st)$ і $s\cdot U_{2}(t)\subseteq U(st)$.

Нехай $(s,a_{1}^{-1}a_{2}),(t,b_{1}^{-1}b_{2})$ -- довільні елементи напівгрупи $\mathscr{P}_{\lambda}(\theta,S)$. Розглянемо можливі випадки.

1. Якщо існує слово $u\in\lambda^*$ таке, що $b_{1}=ua_{2}$, то
\begin{equation*}
(s,a_{1}^{-1}a_{2})*(t,b_{1}^{-1}b_{2})=(\theta^{|u|}(s)t,a_{1}^{-1}u^{-1}b_{2}).
\end{equation*}
Нехай $W_{a_{1}^{-1}u^{-1}b_{2}}$~--- довільний відкритий окіл точки $(\theta^{|u|}(s)t,a_{1}^{-1}u^{-1}b_{2})$ у тополо\-гіч\-ному просторі  $(\mathscr{P}_{\lambda}(\theta,S),\tau_{\texttt{s\!t}})$, де $W$ -- відкритий окіл точки $\theta^{|u|}(s)t$ в $(S,\tau_S)$. Тоді з нарізної неперервності напівгрупової операції в $(S,\tau_S)$ і з неперервності гомомор\-фіз\-му $\theta$ випливає, що існують відкриті околи $V(s)$ i $V(t)$ точок $s$ і $t$ у просторі $(S,\tau_S)$, відповідно, такі, що $\theta^{|u|}(V(s))\cdot t\subseteq W$ i $\theta^{|u|}(s)\cdot V(t)\subseteq W$. Тоді
\begin{equation*}
  V(s)_{a_{1}^{-1}a_{2}}*(t,b_{1}^{-1}b_{2}) \subseteq W_{a_{1}^{-1}u^{-1}b_{2}} \qquad \hbox{i} \qquad (s,a_{1}^{-1}a_{2})*V(t)_{b_{1}^{-1}b_{2}}\subseteq W_{a_{1}^{-1}u^{-1}b_{2}}.
\end{equation*}

2. Якщо існує слово $v\in\lambda^*$ таке, що $a_{2}=vb_{1}$, то
\begin{equation*}
(s,a_{1}^{-1}a_{2})*(t,b_{1}^{-1}b_{2})=(s\theta^{|v|}(t),a_{1}^{-1}vb_{2}).
\end{equation*}
Нехай $W_{a_{1}^{-1}vb_{2}}$~--- довільний відкритий окіл точки $(s\theta^{|v|}(t),a_{1}^{-1}vb_{2})$ у топологічному просторі  \linebreak $(\mathscr{P}_{\lambda}(\theta,S),\tau_{\texttt{s\!t}})$, де $W$ -- відкритий окіл точки $s\theta^{|v|}(t)$ в $(S,\tau_S)$. Тоді з на\-різ\-ної неперервності напівгрупової операції в $(S,\tau_S)$ і з неперервності гомоморфізму $\theta$ випливає, що існують відкриті околи $V(s)$ i $V(t)$ точок $s$ і $t$ в $(S,\tau_S)$, відповідно, такі, що $V(s) \cdot\theta^{|v|}(t)\subseteq W$ i $s\cdot \theta^{|v|}(V(t))\subseteq W$. Тоді
\begin{equation*}
  V(s)_{a_{1}^{-1}a_{2}}*(t,b_{1}^{-1}b_{2}) \subseteq W_{a_{1}^{-1}vb_{2}} \qquad \hbox{i} \qquad (s,a_{1}^{-1}a_{2})*V(t)_{b_{1}^{-1}b_{2}}\subseteq W_{a_{1}^{-1}vb_{2}}.
\end{equation*}

3. Якщо $(s,a_{1}^{-1}a_{2})*(t,b_{1}^{-1}b_{2})=\boldsymbol{0}$, то
\begin{equation*}
V(s)_{a_{1}^{-1}a_{2}}*(t,b_{1}^{-1}b_{2})=\{\boldsymbol{0}\} \qquad \hbox{і} \qquad (s,a_{1}^{-1}a_{2})*V(t){b_{1}^{-1}b_{2}}=\{\boldsymbol{0}\}
\end{equation*}
для довільних відкритих околів $V(s)$ i $V(t)$ точок $s$ i $t$, відповідно, в $(S,\tau_S)$.

\smallskip

Також, маємо, що $V(s)_{a_{1}^{-1}a_{2}}*\{\boldsymbol{0}\}=\{\boldsymbol{0}\}$, $\{\boldsymbol{0}\}*V(s)_{a_{1}^{-1}a_{2}}=\{\boldsymbol{0}\}$ i $\{\boldsymbol{0}\}*\{\boldsymbol{0}\}=\{\boldsymbol{0}\}$ для довільного відкритого околу $V(s)$ точки $s$ в $(S,\tau_S)$.
\end{proof}

\begin{remark}\label{remark-3.2}
\begin{itemize}
  \item[$(i)$] Оскільки топологічний простір $(\mathscr{P}_{\lambda}(\theta,S),\tau_{\texttt{s\!t}})$ є топологіч\-ною сумою $\omega\cdot\lambda$ копій простору $(S,\tau_S)$ та ізольованої точки $\{\boldsymbol{0}\}$, то напівтопологічна напівгрупа $(\mathscr{P}_{\lambda}(\theta,S),\tau_{\texttt{s\!t}})$ успадковує усі властивості простору $(S,\tau_S)$, що зберігаються нескінченною топологічною сумою топологічних просторів. Зокрема, метрику $d_S$ з $S$ на $\mathscr{P}_{\lambda}(\theta,S)$ можна продовжити так:
      \begin{equation*}
        d_{\texttt{s\!t}}((s,a_{1}^{-1}a_{2}),(t,b_{1}^{-1}b_{2}))=
        \left\{
          \begin{array}{cl}
            d(s,t), & \hbox{якщо~} a_{1}^{-1}a_{2}=b_{1}^{-1}b_{2};\\
            1, & \hbox{в іншому випадку.}
          \end{array}
        \right.
      \end{equation*}
      Очевидно, що топологія, породжена метрикою $d_{\texttt{s\!t}}$ на $\mathscr{P}_{\lambda}(\theta,S)$, збігається з топологією $\tau_{\texttt{s\!t}}$.
  \item[$(ii)$] Твердження~\ref{proposition-3.1} виконується для топологічних і топологічних інверсних напівгруп.
\end{itemize}
\end{remark}

\begin{proposition}\label{proposition-3.3}
Нехай $(\mathscr{P}_{\lambda}(\theta,S),\tau)$ -- напівтопологічна напівгрупа. Тоді для довільних слів $a_{1},a_{2},b_{1},b_{2},\in \lambda^*$ топологічні підпростори $S_{a_{1}^{-1}a_{2}}$ і $S_{b_{1}^{-1}b_{2}}$ є гомео\-морф\-ними, а $S_{a^{-1}_{1}a_{1}}$ та $S_{b^{-1}_{1}b_{1}}$ -- топологічно ізоморфні піднапівгрупи в $(\mathscr{P}_{\lambda}(\theta,S),\tau)$.
\end{proposition}

\begin{proof}
Означимо відображення $\phi_{a_{1}^{-1}a_{2}}^{b_{1}^{-1}b_{2}}:\mathscr{P}_{\lambda}(\theta,S)\rightarrow\mathscr{P}_{\lambda}(\theta,S)$ та $\phi_{b_{1}^{-1}b_{2}}^{a_{1}^{-1}a_{2}}:\mathscr{P}_{\lambda}(\theta,S)\rightarrow\mathscr{P}_{\lambda}(\theta,S)$ формулами
\begin{equation*}
\phi_{a_{1}^{-1}a_{2}}^{b_{1}^{-1}b_{2}}(s)=(1_{S},b_{1}^{-1}a_{1})*s*(1_{S},a_{2}^{-1}b_{2}) \qquad \hbox{i} \qquad \phi_{b_{1}^{-1}b_{2}}^{a_{1}^{-1}a_{2}}(s)=(1_{S},a_{1}^{-1}b_{1})*s*(1_{S},b_{2}^{-1}a_{2}).
\end{equation*}
Відображення $\phi_{a_{1}^{-1}a_{2}}^{b_{1}^{-1}b_{2}}$ та $\phi_{b_{1}^{-1}b_{2}}^{a_{1}^{-1}a_{2}}$ є неперервними як композиції зсувів у напівтопологічній напівгрупі $(\mathscr{P}_{\lambda}(\theta,S),\tau)$, а також виконуються рівності $\phi_{b_{1}^{-1}b_{2}}^{a_{1}^{-1}a_{2}}(\phi_{a_{1}^{-1}a_{2}}^{b_{1}^{-1}b_{2}}(s))=s$ і $\phi_{a_{1}^{-1}a_{2}}^{b_{1}^{-1}b_{2}}(\phi_{b_{1}^{-1}b_{2}}^{a_{1}^{-1}a_{2}}(t))=t$ для довільних $s\in S_{a_{1}^{-1}a_{2}}$ і $t\in S_{b_{1}^{-1}b_{2}}$, а отже, їх звуження $\phi_{a_{1}^{-1}a_{2}}^{b_{1}^{-1}b_{2}}|_{S_{a_{1}^{-1}a_{2}}}$ i $(\phi_{b_{1}^{-1}b_{2}}^{a_{1}^{-1}a_{2}})|_{S_{a_{1}^{-1}a_{2}}}$ є гомеоморфізмами підпросторів $S_{a_{1}^{-1}a_{2}}$ і $S_{b_{1}^{-1}b_{2}}$ в $(\mathscr{P}_{\lambda}(\theta,S),\tau)$. У випадку піднапівгруп $S_{a^{-1}_{1}a_{1}}$ та $S_{b^{-1}_{1}b_{1}}$, очевидно, що відображення $\phi_{a^{-1}_{1}a_{1}}^{b^{-1}_{1}b_{1}}|_{S_{a^{-1}_{1}a_{1}}}$ є ізоморфізмом.
\end{proof}

\begin{lemma}\label{lemma-3.4}
Нехай $\tau$~--- гаусдорфова трансляційно неперервна топологія на напівгрупі $\mathscr{P}_{\lambda}(\theta,S)$. Тоді:
\begin{itemize}
  \item[$(i)$] для довільної точки $(s,\varepsilon)$ існує її відкритий окіл $V\subseteq S_{\varepsilon}$;
  \item[$(ii)$] для довільного елемента $s\in S$ та довільного слова $w\in\lambda^*$ існує відкритий окіл $V$ точки $(s,w)$ такий, що $V\subseteq S_{w}$;
  \item[$(iii)$] для довільного елемента $s\in S$ і довільного слова $w\in\lambda^*$ існує відкритий окіл $V$ точки $(s,w^{-1})$ такий, що $V\subseteq S_{w^{-1}}$;
  \item[$(iv)$] для довільного елемента $s\in S$ і довільних непорожніх слів $u,v\in\lambda^*$ існує відкритий окіл $V$ точки $(s,u^{-1}v)$, що містить лише точки вигляду $(t,x^{-1}y)$, де $x$~--- суфікс слова $u$, a $y$~--- суфікс слова $v$.
\end{itemize}
\end{lemma}

\begin{proof}
$(i)$ Зафіксуємо довільну літеру $x\in\lambda$. Тоді
\begin{equation*}
  (s,\varepsilon)*(1_s,x^{-1}x)=(\theta(s),x^{-1}x) \qquad \hbox{i} \qquad (1_s,x^{-1}x)*(s,\varepsilon)=(\theta(s),x^{-1}x).
\end{equation*}
Нехай $W$~--- відкритий окіл точки $(\theta(s),x^{-1}x)$, що не містить нуля $\boldsymbol{0}$. З нарізної неперервності напівгрупової операції в $(\mathscr{P}_{\lambda}(\theta,S),\tau)$ випливає, що існує відкритий окіл $V$ точки $(s,\varepsilon)$ такий, що $V*(1_s,x^{-1}x)\subseteq W$ i $(1_s,x^{-1}x)*V\subseteq W$. Оскільки $(1_s,x^{-1}x)$~--- ідемпотент напівгрупи $\mathscr{P}_{\lambda}(\theta,S)$, то з гаусдорфовості простору $(\mathscr{P}_{\lambda}(\theta,S),\tau)$ випливає, що
\begin{equation*}
A_x=(1_s,x^{-1}x)*\mathscr{P}_{\lambda}(\theta,S)\cup \mathscr{P}_{\lambda}(\theta,S)*(1_s,x^{-1}x)
\end{equation*}
--- замкнена підмножина в $(\mathscr{P}_{\lambda}(\theta,S),\tau)$. Звідси випливає, що, не зменшуючи загальності, можемо вважати, що  $V\subseteq \mathscr{P}_{\lambda}(\theta,S)\setminus A_x$.

Зафіксуємо довільний елемент $(t,c^{-1}d)\in V$. Зауважимо, що
\begin{equation*}
  (t,c^{-1}d)=(t,c^{-1}d)*(1_S,d^{-1}d) \qquad \hbox{i} \qquad (t,c^{-1}d)=(1_S,c^{-1}c)*(t,c^{-1}d).
\end{equation*}
З напівгрупової операції \eqref{eq-1} визначеної на $\mathscr{P}_{\lambda}(\theta,S)$ випливає, що $x\notin\operatorname{suff}(d)$ i $x\notin\operatorname{suff}(c)$. Отож, якщо $c\neq\varepsilon$ і $d\neq\varepsilon$, і врахувавши, що $x$~--- літера алфавіту $\lambda$, то отримуємо рівності
\begin{equation*}
  (1_S,d^{-1}d)*(1_S,x^{-1}x)=\boldsymbol{0} \qquad \hbox{i} \qquad (1_S,c^{-1}c)*(1_S,x^{-1}x)=\boldsymbol{0},
\end{equation*}
з яких випливає, що
\begin{equation*}
(t,c^{-1}d)*(1_S,x^{-1}x)=(t,c^{-1}d)*(1_S,d^{-1}d)*(1_S,x^{-1}x)=(t,c^{-1}d)*\boldsymbol{0}=\boldsymbol{0}\in W
\end{equation*}
і
\begin{equation*}
(1_S,x^{-1}x)*(t,c^{-1}d)=(1_S,x^{-1}x)*(1_S,c^{-1}c)*(t,c^{-1}d)=\boldsymbol{0}*(t,c^{-1}d)=\boldsymbol{0}\in W.
\end{equation*}
Отримали протиріччя. Отож $c=d=\varepsilon$, а отже, $V\subseteq S_{\varepsilon}$.

\smallskip

$(ii)$ За твердженням $(i)$ для точки $(s,\varepsilon)$ існує її відкритий окіл $W\subseteq S_{\varepsilon}$. Оскільки $(s,w)*(1_S,w^{-1})=(s,\varepsilon)$ для довільного слова $w\in\lambda^*$, то з нарізної неперервності напівгрупової операції в $(\mathscr{P}_{\lambda}(\theta,S),\tau)$ випливає, що існує відкритий окіл $V$ точки $(s,w)$ такий, що $V*(1_S,w^{-1})\subseteq W$. Якщо $(t,c^{-1}d)\in V$, то з рівності
\begin{equation*}
  (t,c^{-1}d)*(1_S,w^{-1})=
\left\{
  \begin{array}{cl}
    (t,c^{-1}d_1),                      & \hbox{якщо~} w\in\operatorname{suff^{o}}(d) \hbox{~i~} d=d_1w;\\
    (t,c^{-1}),                         & \hbox{якщо~} d=w;\\
    (\theta^{|w_1|}(t),c^{-1}w_1^{-1}), & \hbox{якщо~} d\in\operatorname{suff^{o}}(w) \hbox{~i~} w=w_1d;\\
    \boldsymbol{0}, & \hbox{в інших випадках}
  \end{array}
\right.
\end{equation*}
та включення $W\subseteq S_{\varepsilon}$ отримуємо, що $c=d_1=w_1=\varepsilon$, а отже, $d=w$. Звідки випливає, що $V\subseteq S_{w}$.

\smallskip

Доведення твердження $(iii)$  аналогічне до $(ii)$.

\smallskip

$(iv)$ Зафіксуємо довільний елемент $(s,u^{-1}v)$ напівгрупи $\mathscr{P}_{\lambda}(\theta,S)$. Очевидно, що виконуються рівності
\begin{equation*}
  (s,u^{-1}v)*(1_S,v^{-1})=(s,u^{-1}) \qquad \hbox{i} \qquad (1_S,u)*(s,u^{-1}v)=(s,v).
\end{equation*}
З тверджень $(ii)$ і $(iii)$ випливає, що існують відкриті околи $W_{(s,u^{-1})}$ i $W_{(s,v)}$ точок $(s,u^{-1})$ i $(s,v)$ в топологічному просторі $(\mathscr{P}_{\lambda}(\theta,S),\tau)$, відповідно, такі, що $W_{(s,u^{-1})}\subseteq S_{u^{-1}}$ i $W_{(s,v)}\subseteq S_{v}$. З нарізної неперервності напівгрупової операції в $(\mathscr{P}_{\lambda}(\theta,S),\tau)$ випливає, що існує відкритий окіл $V_{(s,u^{-1}v)}$ точки $(s,u^{-1}v)$ такий, що
\begin{equation*}
  V_{(s,u^{-1}v)}*(1_S,v^{-1})\subseteq W_{(s,u^{-1})} \qquad \hbox{i} \qquad (1_S,u)*V_{(s,u^{-1}v)}\subseteq W_{(s,v)}.
\end{equation*}
Зафіксуємо довільну точку $(t,x^{-1}y)\in V_{(s,u^{-1}v)}$. Тоді
\begin{equation*}
(t,x^{-1}y)*(1_S,v^{-1})=(t_1,u^{-1}) \qquad \hbox{i} \qquad (1_S,u)*(t,x^{-1}y)=(t_2,v),
\end{equation*}
для деяких $t_1,t_2\in S$.
З кожної з цих рівностей випливає, що $x$ є суфіксом слова $u$, a $y$ є суфіксом слова $v$.
\end{proof}

Для довільних слів $u,v\in\lambda^*$ позначимо
\begin{equation*}
  \mathscr{P}^{[u^{-1}v]}_{\lambda}(\theta,S)=\left\{(s,u^{-1}a^{-1}bv)\in\mathscr{P}_{\lambda}(\theta,S) \colon s\in S, \; a,b\in \lambda^*\right\}.
\end{equation*}

\begin{lemma}\label{lemma-3.5}
Нехай $\tau$~--- гаусдорфова трансляційно неперервна топологія на напівгрупі $\mathscr{P}_{\lambda}(\theta,S)$. Тоді для довільних слів $u,v\in\lambda^*$ підпростір $\mathscr{P}^{[u^{-1}v]}_{\lambda}(\theta,S)$ в $(\mathscr{P}_{\lambda}(\theta,S),\tau)$ гомеоморфний простору $(\mathscr{P}_{\lambda}(\theta,S),\tau)$.
\end{lemma}

\begin{proof}
З нарізної неперервності напівгрупової операції в $(\mathscr{P}_{\lambda}(\theta,S),\tau)$ випливає, що відображення
\begin{equation*}
f\colon \mathscr{P}_{\lambda}(\theta,S)\to \mathscr{P}^{[u^{-1}v]}_{\lambda}(\theta,S), \; x\mapsto (1_S,u^{-1})*x*(1_S,v)
\end{equation*}
i
\begin{equation*}
h\colon \mathscr{P}^{[u^{-1}v]}_{\lambda}(\theta,S)\to \mathscr{P}_{\lambda}(\theta,S), \; x\mapsto (1_S,u)*x*(1_S,v^{-1})
\end{equation*}
є неперервними. Очевидно, що $f\circ h$ i $h\circ f$~--- тотожні відображення множин $\mathscr{P}^{[u^{-1}v]}_{\lambda}(\theta,S)$ і $\mathscr{P}_{\lambda}(\theta,S)$, відповідно, звідки випливає твердження леми.
\end{proof}

З лем~\ref{lemma-3.4} i \ref{lemma-3.5} випливає наслідок \ref{corollary-3.6}.

\begin{corollary}\label{corollary-3.6}
Нехай $u,v\in\lambda^*$ i $\tau$~--- гаусдорфова трансляційно неперервна топологія на напівгрупі $\mathscr{P}_{\lambda}(\theta,S)$. Тоді:
\begin{itemize}
  \item[$(i)$] для довільної точки $(s,u^{-1}v)$ існує її відкритий окіл $V\cap\mathscr{P}^{[u^{-1}v]}_{\lambda}(\theta,S)\subseteq S_{u^{-1}v}$;
  \item[$(ii)$] для довільного елемента $s\in S$ та довільного слова $w\in\lambda^*$ існує відкритий окіл $V$ точки $(s,u^{-1}wv)$ такий, що $V\cap \mathscr{P}^{[u^{-1}v]}_{\lambda}(\theta,S)\subseteq S_{u^{-1}wv}$;
  \item[$(iii)$] для довільного елемента $s\in S$ та довільного слова $w\in\lambda^*$ існує відкритий окіл $V$ точки $(s,u^{-1}w^{-1}v)$ такий, що $V\cap \mathscr{P}^{[u^{-1}v]}_{\lambda}(\theta,S)\subseteq S_{u^{-1}w^{-1}v}$.
\end{itemize}
\end{corollary}

\begin{lemma}\label{lemma-3.7}
Нехай $\tau$~--- гаусдорфова трансляційно неперервна топологія на напівгрупі $\mathscr{P}_{\lambda}(\theta,S)$ така, що $S_{u^{-1}v}$~--- замкнена підмножина в топологічному просторі $(\mathscr{P}_{\lambda}(\theta,S),\tau)$ для деяких непорожніх слів $u,v\in\lambda^*$. Тоді  $S_{u^{-1}}$ i $S_{v}$~--- замкнені підмножини в $(\mathscr{P}_{\lambda}(\theta,S),\tau)$.
\end{lemma}

\begin{proof}
Очевидно, що для довільного елемента $s\in S$ виконується рівність
\begin{equation*}
  (s,u^{-1})*(1_S,v)=(s,u^{-1}v).
\end{equation*}
Нехай $(t,x^{-1}y)$~--- довільний елемент напівгрупи $\mathscr{P}_{\lambda}(\theta,S)$ такий, що
\begin{equation*}
(t,x^{-1}y)*(1_S,v)\in S_{u^{-1}v}.
\end{equation*}
Тоді
\begin{align*}
  (t,x^{-1}y)*(1_S,v)&=(t\cdot\theta^{|y|}(1_S),x^{-1}yv)=\\
  &=(t\cdot 1_S,x^{-1}yv)=\\
  &=(t,x^{-1}yv),
\end{align*}
а отже, $y=\varepsilon$ i $x=u$. Отож повним прообразом множини $S_{u^{-1}v}$ стосовно правого зсуву на елемент $(1_S,v)$ є множина $S_{u^{-1}}$. Оскільки зсуви в $(\mathscr{P}_{\lambda}(\theta,S),\tau)$ неперервні відображення, то за теоремою 1.4.1 з \cite{Engelking-1989}, $S_{u^{-1}}$~--- замкнена підмножина в тополо\-гіч\-ному просторі $(\mathscr{P}_{\lambda}(\theta,S),\tau)$.

Аналогічно, для довільного елемента $s\in S$ виконується рівність
\begin{equation*}
  (1_S,u^{-1})*(s,v)=(s,u^{-1}v).
\end{equation*}
Якщо $(t,x^{-1}y)$~--- довільний елемент напівгрупи $\mathscr{P}_{\lambda}(\theta,S)$ такий, що
\begin{equation*}
(1_S,u^{-1})*(t,x^{-1}y)\in S_{u^{-1}v},
\end{equation*}
то
\begin{align*}
  (1_S,u^{-1})*(t,x^{-1}y)&=(\theta^{|x|}(1_S)\cdot t,u^{-1}x^{-1}y)=\\
  &=(1_S\cdot t,u^{-1}x^{-1}y)=\\
  &=(t,u^{-1}x^{-1}y),
\end{align*}
а отже, $y=v$ i $x=\varepsilon$. Звідси випливає, що повним прообразом множини $S_{u^{-1}v}$ стосовно лівого зсуву на елемент $(1_S,u^{-1})$ є множина $S_{v}$. Далі знову з теореми 1.4.1 \cite{Engelking-1989} випливає, що $S_{v}$~--- замкнена підмножина в  $(\mathscr{P}_{\lambda}(\theta,S),\tau)$.
\end{proof}

\begin{lemma}\label{lemma-3.8}
Нехай $\tau$~--- гаусдорфова трансляційно неперервна топологія на напівгрупі $\mathscr{P}_{\lambda}(\theta,S)$ така, що $S_{u^{-1}}$~--- замкнена підмножина в топологічному просторі $(\mathscr{P}_{\lambda}(\theta,S),\tau)$  для деякого непорожнього слова $u\in\lambda^*$. Тоді $S_{\varepsilon}$~--- замкнена підмножина в $(\mathscr{P}_{\lambda}(\theta,S),\tau)$.
\end{lemma}

\begin{proof}
Очевидно, що для довільного елемента $s\in S$ виконується рівність
\begin{equation*}
  (1_S,u^{-1})*(s,\varepsilon)=(s,u^{-1}).
\end{equation*}
Нехай $(t,x^{-1}y)$~--- довільний елемент напівгрупи $\mathscr{P}_{\lambda}(\theta,S)$ такий, що
\begin{equation*}
(1_S,u^{-1})*(t,x^{-1}y)\in S_{u^{-1}}.
\end{equation*}
Тоді
\begin{align*}
  (1_S,u^{-1})*(t,x^{-1}y)&=(\theta^{|x|}(1_S)\cdot t,u^{-1}x^{-1}y)=\\
  &=(1_S\cdot t,u^{-1}x^{-1}y)=\\
  &=(t,u^{-1}x^{-1}y),
\end{align*}
а отже, $x=y=\varepsilon$. Звідси випливає, що повним прообразом множини $S_{u^{-1}}$ стосовно лівого зсуву на елемент $(1_S,u^{-1})$ є множина $S_{\varepsilon}$. Оскільки зсуви в $(\mathscr{P}_{\lambda}(\theta,S),\tau)$ неперервні відображення, то за теоремою 1.4.1 з \cite{Engelking-1989}, $S_{\varepsilon}$~--- замкнена підмножина в топологічному просторі $(\mathscr{P}_{\lambda}(\theta,S),\tau)$.
\end{proof}

\begin{lemma}\label{lemma-3.9}
Нехай $\tau$~--- гаусдорфова трансляційно неперервна топологія на напівгрупі $\mathscr{P}_{\lambda}(\theta,S)$ така, що $S_{v}$~--- замкнена підмножина в топологічному просторі $(\mathscr{P}_{\lambda}(\theta,S),\tau)$ для деякого непорожнього слова $v\in\lambda^*$. Тоді $S_{\varepsilon}$~--- замкнена підмножина в $(\mathscr{P}_{\lambda}(\theta,S),\tau)$.
\end{lemma}

\begin{proof}
З означення напівгрупової операції в $\mathscr{P}_{\lambda}(\theta,S)$ випливає, що для довільного елемента $s\in S$ виконується рівність
\begin{equation*}
  (s,\varepsilon)*(1_S,v)=(s,v).
\end{equation*}
Нехай $(t,x^{-1}y)$~--- довільний елемент напівгрупи $\mathscr{P}_{\lambda}(\theta,S)$ такий, що
\begin{equation*}
(t,x^{-1}y)*(1_S,v)\in S_{v}.
\end{equation*}
Тоді
\begin{align*}
  (t,x^{-1}y)*(1_S,v)&=(t\cdot\theta^{|y|}(1_S),x^{-1}yv)=\\
  &=(t\cdot 1_S,x^{-1}yv)=\\
  &=(t,x^{-1}yv),
\end{align*}
а отже, $x=y=\varepsilon$. Отож повним прообразом множини $S_{v}$ стосовно правого зсуву на елемент $(1_S,v)$ є множина $S_{\varepsilon}$. Оскільки зсуви в $(\mathscr{P}_{\lambda}(\theta,S),\tau)$ є неперервними, то за теоремою 1.4.1 з \cite{Engelking-1989}, $S_{\varepsilon}$~--- замкнена підмножина в топологічному просторі $(\mathscr{P}_{\lambda}(\theta,S),\tau)$.
\end{proof}

\begin{theorem}\label{theorem-3.10}
Нехай $\tau$~--- гаусдорфова трансляційно неперервна топологія на напівгрупі $\mathscr{P}_{\lambda}(\theta,S)$ така, що $S_{\varepsilon}$~--- замкнена підмножина в топологічному просторі $(\mathscr{P}_{\lambda}(\theta,S),\tau)$. Тоді $S_{u^{-1}v}$~--- замк\-нена підмножина в $(\mathscr{P}_{\lambda}(\theta,S),\tau)$ для довільних слів $u,v\in\lambda^*$.
\end{theorem}

\begin{proof}
Теорему доведемо індукцією по довжині слів $u,v\in\lambda^*$.

Спочатку доведемо, що $S_{u^{-1}v}$~--- замкнена підмножина в $(\mathscr{P}_{\lambda}(\theta,S),\tau)$ для довільних слів $u,v\in\lambda^*$, довжини яких не перевищують $1$.

Нехай $a$~--- довільна літера алфавіту $\lambda$. З означення напівгрупової операції в $\mathscr{P}_{\lambda}(\theta,S)$ випливає, що для довільного елемента $s\in S$ виконується рівність
\begin{equation*}
  (s,a)*(1_S,a^{-1})=(s\cdot 1_S,\varepsilon)=(s,\varepsilon).
\end{equation*}
Нехай $(t,x^{-1}y)$~--- довільний елемент напівгрупи $\mathscr{P}_{\lambda}(\theta,S)$ такий, що $(t,x^{-1}y)*(1_S,a^{-1})\in S_{\varepsilon}$. Тоді
\begin{equation*}
  (t,x^{-1}y)*(1_S,a^{-1})=
\left\{
  \begin{array}{cllr}
    (\theta(t)),x^{-1}a^{-1}), & \hbox{якщо~} y=\varepsilon;                         && (1_1)\\
    (t,x^{-1}),                & \hbox{якщо~} y=a;                                   && (1_2)\\
    (t,x^{-1}y_1),             & \hbox{якщо~} y=y_1a \hbox{~~та~~} y_1\neq\varepsilon; && (1_3)\\
    \boldsymbol{0},            & \hbox{в інших випадках.}                            && (1_4)
  \end{array}
\right.
\end{equation*}
Очевидно, що випадки $(1_1)$, $(1_3)$ та $(1_4)$ неможливі, а отже, виконується випадок $(1_2)$. Отож отримаємо, що $x=\varepsilon$ i $y=a$. Звідси випливає, що повним прообразом множини $S_{\varepsilon}$ стосовно правого зсуву на елемент $(1_S,a^{-1})$ є множина $S_{a}$. Оскільки зсуви в $(\mathscr{P}_{\lambda}(\theta,S),\tau)$ є неперервними, то за теоремою 1.4.1 з \cite{Engelking-1989}, $S_{a}$~--- замкнена підмножина в топологічному просторі $(\mathscr{P}_{\lambda}(\theta,S),\tau)$.

Далі, для довільної літери $b$ алфавіту $\lambda$, з означення напівгрупової операції в $\mathscr{P}_{\lambda}(\theta,S)$ випливає, що для довільного елемента $s\in S$ виконується рівність
\begin{equation*}
  (1_S,b)*(s,b^{-1})=(1_S\cdot s,\varepsilon)=(s,\varepsilon).
\end{equation*}
Нехай $(t,x^{-1}y)$~--- довільний елемент напівгрупи $\mathscr{P}_{\lambda}(\theta,S)$ такий, що
\begin{equation*}
(1_S,b)*(t,x^{-1}y)\in S_{\varepsilon}.
\end{equation*}
Тоді
\begin{equation*}
  (1_S,b)*(t,x^{-1}y)=
\left\{
  \begin{array}{cllr}
    (\theta(t)),by),  & \hbox{якщо~} x=\varepsilon;                         && (2_1)\\
    (t,y),            & \hbox{якщо~} x=b;                                   && (2_2)\\
    (t,x_1^{-1}y),    & \hbox{якщо~} x=x_1b \hbox{~~та~~} x_1\neq\varepsilon; && (2_3)\\
    \boldsymbol{0},   & \hbox{в інших випадках.}                            && (2_4)
  \end{array}
\right.
\end{equation*}
Очевидно, що випадки $(2_1)$, $(2_3)$ та $(2_4)$ неможливі, а отже, виконується випадок $(2_2)$. Отож маємо, що $x=b$ i $y=\varepsilon$. Звідси випливає, що повним прообразом множини $S_{\varepsilon}$ стосовно лівого зсуву на елемент $(1_S,b)$ є множина $S_{b^{-1}}$. Оскільки зсуви в $(\mathscr{P}_{\lambda}(\theta,S),\tau)$ є неперервними, то за теоремою 1.4.1 з \cite{Engelking-1989}, $S_{b^{-1}}$~--- замкнена підмножина в топологічному просторі $(\mathscr{P}_{\lambda}(\theta,S),\tau)$.

Нехай $a$ та $b$~--- довільні літери алфавіту $\lambda$. З означення напівгрупової операції в $\mathscr{P}_{\lambda}(\theta,S)$ випливає, що для довільного елемента $s\in S$ виконується рівність
\begin{equation*}
  (1_S,b)*(s,b^{-1}a)=(s,a).
\end{equation*}
Нехай $(t,x^{-1}y)$~--- довільний елемент напівгрупи $\mathscr{P}_{\lambda}(\theta,S)$ такий, що
\begin{equation*}
(1_S,b)*(t,x^{-1}y)\in S_{a}.
\end{equation*}
Розглянемо два випадки $a=b$ й $a\neq b$. Припустимо, що $a=b$. Тоді
\begin{equation*}
  (1_S,a)*(t,x^{-1}y)=
\left\{
  \begin{array}{cllr}
    (\theta(t)),ay),  & \hbox{якщо~} x=\varepsilon;                         && (3_1)\\
    (t,y),            & \hbox{якщо~} x=a;                                   && (3_2)\\
    (t,x_1^{-1}y),    & \hbox{якщо~} x=x_1a \hbox{~~та~~} x_1\neq\varepsilon; && (3_3)\\
    \boldsymbol{0},   & \hbox{в інших випадках.}                            && (3_4)
  \end{array}
\right.
\end{equation*}
Очевидно, що випадки $(3_3)$ та $(3_4)$ неможливі, а отже, виконуються випадки $(3_1)$ та $(3_2)$. У випадку $(3_1)$ матимемо, що $x=y=\varepsilon$ та у випадку $(3_2)$ маємо, що $x=y=a$. Звідси випливає, що повним прообразом множини $S_{a}$ стосовно лівого зсуву на елемент $(1_S,a)$ є множина $S_{a^{-1}a}\cup S_{\varepsilon}$. З неперервності зсувів у $(\mathscr{P}_{\lambda}(\theta,S),\tau)$ та теореми 1.4.1 \cite{Engelking-1989} випливає, що $S_{a^{-1}a}\cup S_{\varepsilon}$~--- замкнена підмножина в $(\mathscr{P}_{\lambda}(\theta,S),\tau)$. За лемою \ref{lemma-3.4}$(i)$, $S_{\varepsilon}$~--- відкрита підмножина в $(\mathscr{P}_{\lambda}(\theta,S),\tau)$, а  отже, $S_{a^{-1}a}$~--- замкнена підмножина в топологічному просторі $(\mathscr{P}_{\lambda}(\theta,S),\tau)$.

Припустимо, що $a\neq b$. Тоді
\begin{equation*}
  (1_S,b)*(t,x^{-1}y)=
\left\{
  \begin{array}{cllr}
    (\theta(t)),by),  & \hbox{якщо~} x=\varepsilon;                         && (4_1)\\
    (t,y),            & \hbox{якщо~} x=b;                                   && (4_2)\\
    (t,x_1^{-1}y),    & \hbox{якщо~} x=x_1b \hbox{~~та~~} x_1\neq\varepsilon; && (4_3)\\
    \boldsymbol{0},   & \hbox{в інших випадках.}                            && (4_4)
  \end{array}
\right.
\end{equation*}
Очевидно, що випадки $(4_1)$, $(4_3)$ та $(4_4)$ неможливі, а отже, виконується випадок $(4_2)$. Отож маємо, що $x=b$ i $y=a$. Звідси випливає, що повним прообразом множини $S_{a}$ стосовно лівого зсуву на елемент $(1_S,b)$ є множина $S_{b^{-1}a}$. Оскільки зсуви в $(\mathscr{P}_{\lambda}(\theta,S),\tau)$ є неперервними відображеннями, то за теоремою 1.4.1 з \cite{Engelking-1989}, $S_{b^{-1}a}$~--- замкнена підмножина в топологічному просторі $(\mathscr{P}_{\lambda}(\theta,S),\tau)$.

Тепер доведемо крок індукції: з того, що $S_{u^{-1}v}$~--- замкнена підмножина в $(\mathscr{P}_{\lambda}(\theta,S),\tau)$ для довільних слів $u,v\in\lambda^*$ довжини $<k$ випливає, що $S_{w^{-1}z}$~--- замк\-не\-на підмножина в $(\mathscr{P}_{\lambda}(\theta,S),\tau)$ для довільних слів $w,z\in\lambda^*$ довжини $\leqslant k$.

Нехай $u\in\lambda^*$~--- слово довжини $k$ i $v\in\lambda^*$~--- слово довжини $<k$. Нехай $b\in\lambda$~--- остання літера слова $u$, тобто $u=u_1 b$ для деякого слова $u_1\in\lambda^*$ довжини $k-1$. Тоді
\begin{equation*}
  (1_S,b)*(s,u^{-1}v)=(1_S,b)*(s,b^{-1}u_1^{-1}v)=(s,u_1^{-1}v).
\end{equation*}
Нехай $(t,x^{-1}y)$~--- довільний елемент напівгрупи $\mathscr{P}_{\lambda}(\theta,S)$ такий, що
\begin{equation*}
(1_S,b)*(t,x^{-1}y)\in S_{u_1^{-1}v}.
\end{equation*}
Тоді
\begin{equation*}
  (1_S,b)*(t,x^{-1}y)=
\left\{
  \begin{array}{cllr}
    (\theta(t)),by),  & \hbox{якщо~} x=\varepsilon;                         && (5_1)\\
    (t,y),            & \hbox{якщо~} x=b;                                   && (5_2)\\
    (t,x_1^{-1}y),    & \hbox{якщо~} x=x_1b \hbox{~~та~~} x_1\neq\varepsilon; && (5_3)\\
    \boldsymbol{0},   & \hbox{в інших випадках.}                            && (5_4)
  \end{array}
\right.
\end{equation*}
Оскільки слово $u_1\in\lambda^*$ має довжину $k-1$, то випадки $(5_1)$, $(5_2)$ і $(5_4)$ неможливі, а отже, виконується випадок $(5_3)$. Отож маємо, що $x=u_1b$ i $y=v$. Звідси випливає, що повним прообразом множини $S_{u_1^{-1}v}$ стосовно лівого зсуву на елемент $(1_S,b)$ є множина $S_{u^{-1}v}$. Оскільки зсуви в $(\mathscr{P}_{\lambda}(\theta,S),\tau)$ є неперервними відображеннями, то за теоремою 1.4.1 з \cite{Engelking-1989}, $S_{u^{-1}v}$~--- замкнена підмножина в тополо\-гіч\-ному просторі $(\mathscr{P}_{\lambda}(\theta,S),\tau)$.

Нехай $u\in\lambda^*$~--- слово довжини $\leqslant k$ i $v\in\lambda^*$~--- слово довжини $k$. Нехай $a\in\lambda$~--- остання літера слова $v$, тобто $v=v_1 a$ для деякого слова $v_1\in\lambda^*$ довжини $k-1$. Тоді
\begin{align*}
  (s,u^{-1}v)*(1_S,a^{-1})&=(s,u^{-1}v_1 a)*(1_S,a^{-1})=\\
  &=(s\cdot 1_S,u^{-1}v_1)=\\
  &=(s,u^{-1}v_1).
\end{align*}
Нехай $(t,x^{-1}y)$~--- довільний елемент напівгрупи $\mathscr{P}_{\lambda}(\theta,S)$ такий, що
\begin{equation*}
(t,x^{-1}y)*(1_S,a^{-1})\in S_{u^{-1}v_1}.
\end{equation*}
Тоді
\begin{equation*}
  (t,x^{-1}y)*(1_S,a^{-1})=
\left\{
  \begin{array}{cllr}
    (\theta(t)),x^{-1}a^{-1}), & \hbox{якщо~} y=\varepsilon;                         && (6_1)\\
    (t,x^{-1}),                & \hbox{якщо~} y=a;                                   && (6_2)\\
    (t,x^{-1}y_1),             & \hbox{якщо~} y=y_1a \hbox{~~та~~} y_1\neq\varepsilon; && (6_3)\\
    \boldsymbol{0},            & \hbox{в інших випадках.}                            && (6_4)
  \end{array}
\right.
\end{equation*}
Оскільки слово $v_1\in\lambda^*$ має довжину $k-1$, то випадки $(6_1)$, $(6_2)$ та $(6_4)$ неможливі, а отже, виконується випадок $(6_3)$. Отож маємо, що $x=u$ i $y=v_1a$. Звідси випливає, що повним прообразом множини $S_{u^{-1}v_1}$ стосовно правого зсуву на елемент $(1_S,a^{-1})$ є множина $S_{u^{-1}v}$. Оскільки зсуви в $(\mathscr{P}_{\lambda}(\theta,S),\tau)$ є неперервними відображеннями, то за теоремою 1.4.1 з \cite{Engelking-1989}, $S_{u^{-1}v}$~--- замкнена підмножина в тополо\-гіч\-ному просторі $(\mathscr{P}_{\lambda}(\theta,S),\tau)$. Це завершує доведення кроку індукції, а отже, виконується тверд\-жен\-ня теореми.
\end{proof}

Нехай $(S,\tau_S)$~--- напівтопологічний моноїд і $\theta\colon S\rightarrow H_{S}(1)$ --- неперервний гомоморфізм з $S$ у його групу одиниць $H_{S}(1)$. Будемо говорити, що напівтопологічний моноїд $(\mathscr{P}_{\lambda}(\theta,S),\tau_{\mathscr{P}_{\lambda}})$ є \emph{тополо\-гіч\-ним $\lambda$-поліциклічним розширенням Брука-Рейлі} моноїда $(S,\tau_S)$ із визначеним гомоморфізмом $\theta$ в класі напівтопологічних напівгруп $\mathfrak{STS}$, якщо відображення $\Upsilon\colon (S,\tau_S)\to (\mathscr{P}_{\lambda}(\theta,S),\tau_{\mathscr{P}_{\lambda}})$, означене за формулою $\Upsilon\colon s\mapsto (s,\varepsilon)$, є тополого-алгебричним вкладенням і $(S,\tau_S),(\mathscr{P}_{\lambda}(\theta,S),\tau_{\mathscr{P}_{\lambda}})\in\mathfrak{STS}$. Якщо  напівтопологічна напівгрупа $(S,\tau_S)$ не містить одиниці, то приєднавши одиницю $1_S$ до $(S,\tau_S)$ як ізольовану точку та означивши гомоморфізм $\theta\colon S\rightarrow H_{S}(1), s\mapsto 1$, ми аналогічно, як і в \cite{Gutik-1994}, отримуємо \emph{тополо\-гіч\-не $\lambda$-поліциклічне розширенням Брука} моноїда $(S,\tau_S)$. Зауважимо, що з твердження \ref{proposition-3.3} випливає, що топологічний ізоморфізм  вкладення $\Upsilon\colon (S,\tau_S)\to (\mathscr{P}_{\lambda}(\theta,S),\tau_{\mathscr{P}_{\lambda}})$ можна визначати й  за формулою $\Upsilon\colon s\mapsto (s,w^{-1}w)$, де $w$~--- довільне слово вільного моноїда $\lambda^*$.

З твердження \ref{proposition-3.1} випливає, що для довільної напівтопологічної напівгрупи $(S,\tau_S)$ існує тополо\-гіч\-не $\lambda$-поліциклічне розширення Брука-Рейлі $(\mathscr{P}_{\lambda}(\theta,S),\tau)$ напівгрупи $(S,\tau_S)$ із визначеним гомоморфізмом $\theta$ в класі напівтопологічних напівгруп таке, що для довільних слів $u$ та $v$ вільного моноїда $\lambda^*$ підмножина $S_{u^{-1}v}$ відкрито-замкнена в топологічному просторі $(\mathscr{P}_{\lambda}(\theta,S),\tau)$ і нуль $\boldsymbol{0}$ є ізольованою точкою цього простору. Тому природно виникає таке запитання: \emph{за яких умов на напівтополо\-гіч\-ну напівгрупу $(S,\tau_S)$ усі тополо\-гіч\-ні $\lambda$-поліциклічні розширення Брука-Рейлі $(\mathscr{P}_{\lambda}(\theta,S),\tau)$ напівгрупи $(S,\tau_S)$ мають одну, чи обидві з цих вище перелічених властивостей?}

На завершенні ми даємо часткову відповідь на запитання: за яких умов підмножина $S_{u^{-1}v}$ відкрито-замкнена в довільному тополо\-гіч\-ному $\lambda$-розширенні Брука-Рейлі $(\mathscr{P}_{\lambda}(\theta,S),\tau)$ напівтопологічної напівгрупи $(S,\tau_S)$?

Нагадаємо, що топологічний простір $X$ називається \emph{компактним}, якщо кожне відкрите по\-крит\-тя простору $X$ містить скінченне підпокриття.

Напівтопологічна напівгрупа $S$ називається  $H$-замкненою в класі гаусдорфових напівтополо\-гіч\-них напівгруп $\mathfrak{STS}$, якщо $S\in\mathfrak{STS}$ і кожна напівтопологічна напівгрупа $T\in\mathfrak{STS}$, що містить напівгрупу $S$,  містить напівгрупу $S$ як замкнений підпростір \cite{Gutik-2014}.

\begin{theorem}\label{theorem-3.11}
Нехай $(S,\tau_S)$~--- гаусдорфовий напівтопологічний моноїд, $\theta\colon S\rightarrow H_{S}(1)$ --- неперервний гомоморфізм і $(\mathscr{P}_{\lambda}(\theta,S),\tau)$ --- тополо\-гіч\-не $\lambda$-поліциклічне роз\-ши\-рен\-ня Брука-Рейлі напівгрупи $(S,\tau_S)$ в класі гаусдорфових напівтопологічних напівгруп. Якщо $(S,\tau_S)$ містить лівий (правий, дво\-біч\-ний) ідеал, який є $H$-замкненою напівгрупою в класі гаусдорфових напівтополо\-гіч\-них напівгруп, то для довільних слів $u$ та $v$ вільного моноїда $\lambda^*$ підмножина $S_{u^{-1}v}$ відкрито-замкнена в тополо\-гіч\-ному просторі $(\mathscr{P}_{\lambda}(\theta,S),\tau)$.
\end{theorem}

\begin{proof}
За теоремою \ref{theorem-3.10} нам достатньо довести, що  $S_{\varepsilon}$~--- замкнена підмножина в топологічному просторі $(\mathscr{P}_{\lambda}(\theta,S),\tau)$, оскільки в цьому випадку з леми \ref{lemma-3.4} випливає, що всі множини вигляду $S_{u^{-1}v}$ відкриті в топологічному просторі $(\mathscr{P}_{\lambda}(\theta,S),\tau)$.

Припустимо, що напівгрупа $(S,\tau_S)$ містить лівий ідеал $I$, що є $H$-замкненою напівгрупою в класі гаусдорфових напівтополо\-гіч\-них напівгруп. Тоді
\begin{equation*}
I_\varepsilon=\left\{(s,\varepsilon)\colon s\in I\right\}
\end{equation*}
є замкненою піднапівгрупою в напівтопологічній напівгрупі $(\mathscr{P}_{\lambda}(\theta,S),\tau)$. Зафіксуємо довільний елемент $(s_0,\varepsilon)\in I_\varepsilon$. Очевидно, що $(t,\varepsilon)*(s_0,\varepsilon)\in I_\varepsilon$ для довільного елемента $t\in S$.
Нехай $(t,x^{-1}y)$~--- довільний елемент напівгрупи $\mathscr{P}_{\lambda}(\theta,S)$ такий, що $(t,x^{-1}y)*(s_0,\varepsilon)\in S_{\varepsilon}$.
Тоді
\begin{equation*}
  (t,x^{-1}y)*(s_0,\varepsilon)=
\left\{
  \begin{array}{cllr}
    (ts_0,\varepsilon),          & \hbox{якщо~} x=y=\varepsilon;                                && (7_1)\\
    (ts_0,x^{-1}),               & \hbox{якщо~} x\neq\varepsilon \hbox{~~i~~} y=\varepsilon;    && (7_2)\\
    (t\theta^{|y|}(s),x^{-1}y),  & \hbox{якщо~} x\neq\varepsilon \hbox{~~i~~} y\neq\varepsilon. && (7_3)
  \end{array}
\right.
\end{equation*}
Очевидно, що випадки $(7_2)$ та $(7_3)$ неможливі, а отже, виконується випадок $(7_1)$. Отож маємо, що $x=y=\varepsilon$. Звідси випливає, що повним прообразом множини $I_\varepsilon$ стосовно лівого зсуву на елемент $(s_0,\varepsilon)$ є множина $S_{\varepsilon}$. Оскільки зсуви в $(\mathscr{P}_{\lambda}(\theta,S),\tau)$ є неперервними відображеннями, то за теоремою 1.4.1 з \cite{Engelking-1989}, $S_{\varepsilon}$~--- замкнена підмножина в топологічному просторі $(\mathscr{P}_{\lambda}(\theta,S),\tau)$. Далі скористаємося теоре\-мою~\ref{theorem-3.10}.

Доведення у випадку правого, чи двобічного ідеалу, аналогічне.
\end{proof}

З теоре\-ми~\ref{theorem-3.11} випливає наслідок \ref{corollary-3.12}.

\begin{corollary}\label{corollary-3.12}
Нехай $(S,\tau_S)$~--- гаусдорфовий напівтопологічний моноїд, $\theta\colon S\rightarrow H_{S}(1)$ --- неперервний гомоморфізм і $(\mathscr{P}_{\lambda}(\theta,S),\tau)$ --- тополо\-гіч\-не $\lambda$-поліциклічне розширення Брука-Рейлі напівгрупи $(S,\tau_S)$ в класі гаусдорфових напівтопологічних напівгруп. Якщо $(S,\tau_S)$ містить компактний лівий (правий, дво\-біч\-ний) ідеал,  то для довільних слів $u$ та $v$ вільного моноїда $\lambda^*$ підмножина $S_{u^{-1}v}$ відкрито-замкнена в топологічному просторі $(\mathscr{P}_{\lambda}(\theta,S),\tau)$.
\end{corollary}

\begin{theorem}\label{theorem-3.13}
Нехай $(S,\tau_S)$~--- гаусдорфовий топологічний інверсний моноїд, $\theta\colon S\rightarrow H_{S}(1)$ --- неперервний гомоморфізм і $(\mathscr{P}_{\lambda}(\theta,S),\tau)$ --- тополо\-гіч\-не $\lambda$-поліциклічне розширення Брука-Рейлі напівгрупи $(S,\tau_S)$ в класі гаусдорфових топологічних інверсних напівгруп. Якщо напівгрупа $S$ містить мінімальний ідемпотент, то для довільних слів $u$ та $v$ вільного моноїда $\lambda^*$ підмножина $S_{u^{-1}v}$ відкрито-замкнена в топологічному просторі $(\mathscr{P}_{\lambda}(\theta,S),\tau)$.
\end{theorem}

\begin{proof}
Оскільки в топологічній інверсній напівгрупі $S$ інверсія та напівгрупова операція неперервні, то кожна максимальна підгрупа в $S$, а отже, і кожен $\mathscr{H}$-клас, є замкненою підмножиною в $S$ \cite{Eberhart-Selden-1969}. Звідси випливає, якщо $e_0$~--- мінімальний ідемпотент напівгрупи $S$, то максимальна підгрупа $H(e_0,\varepsilon)$ в $\mathscr{P}_{\lambda}(\theta,S)$, що містить ідемпотент $(e_0,\varepsilon)$, є замкненою підмножиною в топологічному просторі $(\mathscr{P}_{\lambda}(\theta,S),\tau)$. Далі доведення аналогічне до теореми~\ref{theorem-3.11}.
\end{proof}

\begin{remark}
Напівгрупа Брука--Рейлі $\mathscr{B}(S,\theta)$ над моноїдом $S$ (див. \cite[підрозділ~II.5]{Petrich1984}) є, очевидно, піднапівгрупою в $\mathscr{P}_{\lambda}(\theta,S)$. Зауважимо, що з теорем~\ref{theorem-3.10} і~\ref{theorem-3.11} випливають основні результати, отримані в працях \cite{Gutik-1994, Gutik-1997, Gutik-Pavlyk=2009} для топологічних розширень Брука--Рейлі топологічних і напівтопологічних напівгруп.
\end{remark}


\section*{\textbf{Подяка}}

Автори висловлюють щиру подяку  рецензентові за цінні поради.


\begin{thebibliography}{10}

\bibitem{Gutik-1994}
О. В. Гутик,
\emph{Вложениe топологических полугрупп},
Мат. Студії \textbf{3} (1994), 10--14.

\bibitem{Gutik-1997}
О. В. Гутік,
\emph{Про ослаблення топології прямої суми на напівгрупі Брака},
Вісник Львів. ун-ту. Сер. мех.-мат. \textbf{47} (1997), 17--21.

\bibitem{Andersen-1952}
O. Andersen,
\emph{Ein Bericht uber die Struk\-tur abstrakter Halbgruppen},
PhD Thesis. Ham\-burg, 1952.

\bibitem{Bardyla-Gutik-2016}
  S. Bardyla and O. Gutik,
  \emph{On a semitopological polycyclic monoid},
  Algebra Discr. Math. \textbf{21} (2016), no. 2, 163--183.

\bibitem{Bruck-1958}
  R. J. Bruck,
  \emph{A survey of binary systems},
Berlin-G\"{o}ttingen-Heidelberg, VII, Ergebn. Math. \textbf{20} (1958), 185S.

\bibitem{Carruth-Hildebrant-Koch-1983}
J.~H.~Carruth, J.~A.~Hildebrant, and  R.~J.~Koch,
\emph{The theory of topological semigroups}, Vol. I, Marcel
Dekker, Inc., New York and Basel, 1983.

\bibitem{Carruth-Hildebrant-Koch-1986}
J.~H.~Carruth, J.~A.~Hildebrant, and  R.~J.~Koch,
\emph{The theory of topological semigroups},  Vol. II, Marcel Dekker,
Inc., New York and Basel, 1986.

\bibitem{Clifford-Preston-1961}
A.~H.~Clifford and G.~B.~Preston,
\emph{The algebraic theory of semigroups}, Vol. I,
Amer. Math. Soc. Surveys {\bf 7}, Providence, R.I.,  1961.

\bibitem{Clifford-Preston-1967}
A.~H.~Clifford and G.~B.~Preston,
\emph{The algebraic theory of semigroups}, Vol.  II,
Amer. Math. Soc. Surveys {\bf 7}, Pro\-vi\-den\-ce, R.I.,   1967.

\bibitem{Eberhart-Selden-1969}
C. Eberhart and J. Selden,
\emph{ On the closure of the bicyclic semigroup},
Trans. Amer. Math. Soc. {\bf 144} (1969), 115--126. DOI: 10.1090/S0002-9947-1969-0252547-6

\bibitem{Engelking-1989}
R.~Engelking,
\emph{General topology}, 2nd ed., Heldermann, Berlin, 1989.

\bibitem{Green-1951}
J. A. Green,
\emph{On the structure of semigroups},
Ann. Math. Ser. 2 \textbf{54} (1951), no.~1, 163--172. DOI: 10.2307/1969317

\bibitem{Gutik-2014}
O. Gutik,
\emph{On closures in semitopological inverse semigroups with continuous inversion},
Algebra Discrete Math. \textbf{18} (2014), no. 1, 59--85.

\bibitem{Gutik-2018}
O. Gutik,
\emph{On locally compact semitopological $0$-bisimple inverse $\omega$-semigroups},
Topol. Al\-geb\-ra Appl. \textbf{6} (2018), 77--101.
DOI: 10.1515/taa-2018-0008

\bibitem{Gutik-Pavlyk=2009}
O. V. Gutik and K. P. Pavlyk,
\emph{Bruck-Reilly extension of a semitopological semigroups},
Прик\-ладні проблеми механіки і математики \textbf{7} (2009), 66--72.

\bibitem{PKhylynskyi-Gutik=2019}
P. Khylynskyi and O. Gutik,
On Bruck-Reilly $\lambda$-extensions of semigroups,
The XII International Algebraic Conference in Ukraine, July 02--06, 2019, Vinnytsia, Ukraine. Abstracts. Vinnytsia, 2019, P. 51.

\bibitem{Lawson-1998}
M.~Lawson,
\emph{Inverse semigroups. The theory of partial symmetries},
Singapore: World Scientific, 1998.

\bibitem{Lawson-1999}
  M. V. Lawson,
  \emph{The structure of 0-E-unitary inverse semigroups I: the monoid case},
  Proc. Edinburgh Math. Soc. \textbf{42} (1999), no.~3, 497--520.
  DOI: 10.1017/S0013091500020484

\bibitem{Nivat-Perrot-1970}
  M. Nivat and J.-F. Perrot,
  \emph{Une generalisation du monoide bicyclique},
C. R. Acad. Sci., Paris, S\'{e}r. A \textbf{271} (1970), 824--827.

\bibitem{Petrich1984}
M.~Petrich,
\emph{Inverse semigroups},
John Wiley $\&$ Sons, New York, 1984.

\bibitem{Reilly=1966}
N. R. Reilly,
\emph{Bisimple $\omega$-semigroups},
Proc. Glasgow Math. Assoc. \textbf{7} (1966), no. 3, 160--169. \linebreak
DOI: 10.1017/S2040618500035346

\bibitem{Ruppert-1984}
W.~Ruppert,
\emph{Compact semitopological semigroups: an intrinsic theory},
Lect. Notes Math., \textbf{1079}, Springer, Berlin, 1984.
DOI: 10.1007/BFb0073675

\bibitem{Saito-1965}
T.~Sait\^{o},
\emph{Proper ordered inverse semigroups},
Pac. J. Math. \textbf{15} (1965), no. 2, 649--666. \linebreak
DOI: 10.2140/pjm.1965.15.649

\bibitem{Selden=1975}
A. A. Selden,
\emph{Bisimple $\omega$-semigroups in the locally compact setting},
Bogazici Univ. J. Sci. Math. \textbf{3} (1975), 15--77.

\bibitem{Selden=1976}
A. A. Selden,
\emph{On the closure of bisimple $\omega$-semigroup},
Semigroup Forum \textbf{12} (1976), no. 3, 373--379.
DOI: 10.1007/BF02195943

\bibitem{Selden=1977}
A. A. Selden,
\emph{The kernel of the determining endomorphism of a bisimple $\omega$-semigroup},
Semigroup Forum \textbf{14} (1977), no. 3, 265--271.
DOI: 10.1007/BF02194671

\bibitem{Szendrei-1987}
M. B. Szendrei,
\emph{A generalisation of McAlister's P-theorem for E-unitary regular semigroups},
Acta Sci. Math. (Szeged) \textbf{51} (1987), no. 1--2,  229--249.


\bibitem{Warne=1966}
R. J. Warne,
\emph{A class of bisimple inverse semigroups},
Pac. J. Math. \textbf{18} (1966), no. 3, 563--577.
DOI: 10.1215/S0012-7094-67-03481-3

\end{thebibliography}
\end{document}